\documentclass{amsart}
\def\cal#1{\mathcal{#1}}

          \usepackage{amssymb}
          \usepackage{amsmath}
          \usepackage{amsthm}
          \usepackage{amsfonts}
          \usepackage[english]{babel}
          \usepackage[utf8]{inputenc}
          \usepackage{enumerate}
          \usepackage{graphicx}
          \usepackage{url}
          \usepackage{color}

\usepackage{natbib}
\usepackage{epsfig}
\usepackage{ifthen}

\usepackage[mathscr]{eucal}

\newtheorem{Thm}{Theorem}[section]
\numberwithin{equation}{section}
\newtheorem{Lem}[Thm]{Lemma}
\newtheorem{Prop}[Thm]{Proposition}
\newtheorem{Coro}[Thm]{Corollary}

\theoremstyle{definition}

\newcommand{\comment}[1]{}
\newcommand{\ind}{{\bf 1}}

\def\inddd#1{{\ind}_{\left\{#1\right\}}}
\def\indn#1{\{#1_n\}_{n\in\N}}
\newcommand{\proba}{\mathbb P}
\newcommand{\esp}{{\mathbb E}}

\newcommand{\inv}{^{-1}}

\def\cal{\mathcal}
\newcommand{\calA}{{\cal A}}
\newcommand{\calB}{{\cal B}}

\newcommand{\calE}{{\cal E}}
\newcommand{\filF}{{\cal F}}

\newcommand{\calG}{{\cal G}}

\newcommand{\calK}{{\cal K}}

\def\C{{\mathbb C}}

\def\Z{{\mathbb Z}}

\newcommand{\eqnh}{\begin{eqnarray*}}
\newcommand{\eqne}{\end{eqnarray*}}
\newcommand{\eqnhn}{\begin{eqnarray}}
\newcommand{\eqnen}{\end{eqnarray}}
\newcommand{\equh}{\begin{equation}}
\newcommand{\eque}{\end{equation}}

\def\summ#1#2#3{\sum_{#1 = #2}^{#3}}
\def\prodd#1#2#3{\prod_{#1 = #2}^{#3}}
\def\sif#1#2{\sum_{#1=#2}^\infty}

\newcommand{\eqd}{\stackrel{d}{=}}

\def\topp#1{^{(#1)}}

\def\nn#1{{\left\|#1\right\|}}

\def\abs#1{\left|#1\right|}

\def\ccbb#1{\left\{#1\right\}}

\def\pp#1{\left(#1\right)} 
\def\spp#1{(#1)}
\def\bb#1{\left[#1\right]}

\def\mmid{\;\middle\vert\;}

\def\floor#1{\left\lfloor #1 \right\rfloor}
\def\sfloor#1{\lfloor #1 \rfloor}

\def\qmand{\quad\mbox{ and }\quad}

\def\qmwith{\quad\mbox{ with }\quad}
\def\mfa{\mbox{ for all }}
\def\qmfa{\quad\mbox{ for all }\quad}
\def\mmas{\mbox{ as }}

\def\wt#1{\widetilde{#1}}
\def\wb#1{\overline{#1}}
\def\what#1{\widehat{#1}}

\def\weakto{\Rightarrow}

\def\limn{\lim_{n\to\infty}}

\def\limsupn{\limsup_{n\to\infty}}
\def\liminfn{\liminf_{n\to\infty}}

\def\Z{{\mathbb Z}}

\def\R{{\mathbb R}}
\def\Rd{{\mathbb R^d}}
\def\N{{\mathbb N}}

\def\USC{{\rm USC}}
\def\SM{{\rm SM}}

\newcommand{\id}{infinitely divisible}

\newcommand{\eid}{\stackrel{d}{=}}

\title[Extremal theory]{Extremal theory for long range dependent
  infinitely divisible processes }
\author{}

\begin{document}\sloppy
\author{Gennady Samorodnitsky}
\address{
Gennady Samorodnitsky\\
School of Operations Research and Information Engineering\\
Cornell University\\
220 Rhodes Hall\\
Ithaca, NY 14853, USA.
}
\email{gs18@cornell.edu}

\author{Yizao Wang}
\address
{
Yizao Wang\\
Department of Mathematical Sciences\\
University of Cincinnati\\
2815 Commons Way\\
Cincinnati, OH, 45221-0025, USA.
}
\email{yizao.wang@uc.edu}

\thanks{Samorodnitsky's research was partially supported by the NSF grant
  DMS-1506783 and the ARO
grant  W911NF-12-10385 at Cornell University. Wang's research was
partially supported by the NSA grants  H98230-14-1-0318 and
H98230-16-1-0322, and the ARO grant
W911NF-17-1-0006 at University of Cincinnati.}

\date{\today}
\begin{abstract}
We prove limit theorems of an entirely new type for certain long memory
regularly varying stationary \id\ random processes. These theorems
involve multiple 
phase transitions governed by how long the memory is. Apart from one
regime, our results exhibit limits that are not among the
classical extreme value distributions. Restricted to the one-dimensional
case, the distributions we obtain interpolate, in the appropriate
parameter range, the $\alpha$-Fr\'echet distribution and the skewed
$\alpha$-stable  distribution. In general, the limit is  a new family
of stationary and self-similar random sup-measures with parameters
$\alpha\in(0,\infty)$ and $\beta\in(0,1)$, with 
 representations based on intersections of independent $\beta$-stable regenerative sets. The tail of the limit random sup-measure on each interval with finite positive length is regularly varying with index $-\alpha$. 
 The intriguing structure of these random sup-measures is due to
 intersections of independent $\beta$-stable regenerative sets and the
 fact  that the number of such sets intersecting simultaneously increases to infinity as $\beta$ increases to one.
 The results in this paper extend substantially  previous investigations where only $\alpha\in(0,2)$ and $\beta\in(0,1/2)$ have been considered. 
\end{abstract}

\maketitle

\section{Introduction}
Given a stationary process $(X_n)_{n\in\N}$,  we are interested in the
asymptotic behavior of  the maximum
\[
M_n := \max_{i=1,\dots,n}X_i.
\]
After appropriate normalization, what distributions may arise in the limit?
This is a classical question in probability theory with a very long history. In the case that $(X_n)_{n\in\N}$ is a sequence of independent and identically distributed (i.i.d.)~random variables, all possible limits of the weak convergence in the form of
\equh\label{eq:marginal}
\frac{M_n-a_n}{b_n}\weakto Z
\eque have been known since \cite{fisher28limiting} and
\cite{gnedenko43sur}: these form  the family of extreme-value
distributions, consisting of  Fr\'echet, Gumbel and Weibull
types. Furthermore, the functional  extremal limit theorem in the form of
\equh\label{eq:FELT}
\pp{\frac{M_{\floor{nt}}-a_n}{b_n}}_{t\ge 0}\weakto (Z(t))_{t\ge0}
\eque
in an appropriate topological space has also been known since \cite{dwass:1964} and \cite{lamperti:1964}. The limit process $Z$, when non-degenerate, is known as the {\em extremal process}.

If a stationary process $(X_n)_{n\in\N}$ is not a sequence of
i.i.d.~random variables, the extremes can cluster,  and this can
affect the extremal limit theorems for such processes. 
Research along this line has started since the 60s.  A common feature
of many results in the  literature on this topic is the important role
of the so-called {\em extremal index} $\theta\in(0,1]$. When this
index exists, it affects the limit theorems through the fact that,
asymptotically, the limit law of $M_n$ is the same as that of $\wt
M_{\floor{\theta n}}$, the maximum  of $\floor{\theta n}$
i.i.d.~copies of $X_1$, when one uses the same normalization in both
cases. 
This reflects the following  picture of extremes of such processes:
extreme values of the process occur in finite random clusters, the smaller
$\theta$ indicates larger, on average,  cluster size. 
It is also worth noting that for all $\theta\in(0,1]$, the order of the normalization and the limit laws in~\eqref{eq:marginal} and~\eqref{eq:FELT} are the same as in the i.i.d.~case. 
Therefore, one can view processes with
extremal index  $\theta\in(0,1]$ as having, in the appropriate sense,
{\em short memory} (the reasons for this terminology can be found in
\cite{samorodnitsky:2016}). 
 Standard references for extreme value theory on i.i.d.~and weakly
 dependent sequences include
 \cite{leadbetter:lindgren:rootzen:1983,resnick:1987,dehaan:ferreira:2006}. Point-process
 techniques are fundamental and powerful when investigating such
 problems.

There are situations that for the limit theorems of the
types~\eqref{eq:marginal} and~\eqref{eq:FELT} to hold, the
normalization needs to be of a different order, and even the limit may
be different, from the short memory case.   
We refer to the dependence in such examples as strong or  {\em long range dependence}. 
See the recent monograph \cite{samorodnitsky:2016} for more background and recent developments on long range dependence  in terms of limit theorems (not necessarily extremal ones).
The first example of long range dependence in extreme value theory is
for stationary Gaussian processes: \cite{mittal75limit} showed that
when the correlation $r_n$ satisfies $\limn r_n\log n =
\gamma\in(0,\infty)$, the limit law of $M_n$ is Gumbel convoluted with
a Gaussian distribution, in contrast to the case of $\limn r_n\log n =
0$ where the Gumbel distribution arises in the limit, due to
\cite{berman:1964}. However, very few examples of extremes of
stationary non-Gaussian processes with long range dependence have been
discovered since then.  One of the known examples is important for us
in this paper and we will discuss it below. 

A fundamental work is due to \cite{obrien:torfs:vervaat:1990} who, in
the process of identifying all possible limits of extremes of a
sequence of stationary random variables, pointed out that a more
natural and revealing way to investigate extremes is via the {\em
  random sup-measures}. In this framework, for each $n$ one
investigates the  random sup-measure $M_n$ in the form of 
\[
M_n(B) := \max_{k\in nB\cap\N}X_k, B\subset \R_+,
\]
in an appropriate topological space. Then, a limit theorem for $M_n$
entails at least the finite-dimensional convergence part in a
functional extremal limit theorem as in~\eqref{eq:FELT} 
when restricted to all $B$ in the form of $B = [0,t], t\ge
0$. \cite{obrien:torfs:vervaat:1990} showed that all possible 
random sup-measures $\eta$ on $[0,\infty)$ arising as limits 
 starting from a stationary process $(X_n)_{n\in\N}$ are, up to affine
 transforms, stationary and self-similar, in the sense that 
\[
\eta(\cdot) \eid \eta(\cdot+b), b>0 \qmand \eta(a\cdot) \eid a^H\eta(\cdot), a>0
\]
for some $H>0$. They also provided
examples  of such random sup-measures. However, the investigation of
\cite{obrien:torfs:vervaat:1990} does not directly help in
understanding extremal limit theorems under long range dependence.

In this paper, we investigate the extremes of a general class of
stationary infinitely divisible processes whose law is linked to the
law of a certain null-recurrent Markov chain. Two crucial numerical
parameters impact the properties of such infinitely divisible
processes: $\alpha\in(0,\infty)$ and $\beta\in(0,1)$: the parameter 
$\alpha$ corresponds to the regular variation index of the tail of the
marginal distribution, and $\beta$ determines the rate of the
recurrence of the underlying Markov chain (the larger the $\beta$, the
faster the rate) and, as a result, plays an important role in determining
the memory of the infinitely divisible process. 
The extremes of symmetric $\alpha$-stable processes in this class have
been first investigated in 
\cite{samorodnitsky:2004a}, who showed that when $\beta\in(0,1/2)$,
the partial maxima converge weakly to the Fr\'echet distribution,
although under the normalization $b_n = n^{(1-\beta)/\alpha}$ instead
of $n^{1/\alpha}$ used in the i.i.d. case. (Since
  infinitely divisible processes we are considering are heavy-tailed,
we take the shift $a_n = 0$ in all extremal limit theorems.) The different order of
normalization already indicates long range dependence of the
process. Furthermore, it was pointed out in the same paper that when
$\beta\in(1/2,1)$, the dependence was so strong that the partial
maxima were likely not to converge to the Fr\'echet distribution, but
an alternative  limit distribution was not described. 

Further studies of the extrema of this class of processes have
appeared more  recently, still in the symmetric $\alpha$-stable case,
with $\beta\in (0,1/2)$ (though in a different notation). 
In \cite{owada:samorodnitsky:2015a}, it was shown that the limit in
the functional extremal theorem as in~\eqref{eq:FELT} is, up to a
multiplicative constant, a time-changed extremal process, 
\[
\pp{Z_\alpha(t^{1-\beta})}_{t\ge 0},
\]
where $(Z_\alpha(t))_{t\ge 0}$ is the extremal process for a sequence
of i.i.d.~random variables with tail index $\alpha$ (the
$\alpha$-Fr\'echet extremal process). Subsequently, 
\cite{lacaux:samorodnitsky:2016}, established a limit theorem  in
the framework of convergence of random sup-measures, and, up to a
multiplicative constant,  the limit random sup-measure can be
represented as 
\equh\label{eq:LS16}
\eta(\cdot) = \bigvee_{j=1}^\infty U_j\topp\alpha\inddd{\pp{V\topp\beta_j+R\topp\beta_j}\cap\cdot \neq\emptyset},
\eque
where $(U\topp\alpha_j,V\topp\beta_j,R\topp\beta_j)_{j\in\N}$ is a
measurable enumeration of the points of a Poisson point process on
$\R_+\times\R_+\times\filF(\R_+)$ with intensity $\alpha
u^{-\alpha-1}du (1-\beta) v^{-\beta}dv dP_\beta$. Here $\filF(\R_+)$
is the space of closed subsets of $\R_+$ equipped with Fell topology,
and $P_\beta$ is the law of a $\beta$-stable regenerative set, the closure of
the range of a $\beta$-stable subordinator, on $\filF(\R_+)$. Then 
\[
(\eta([0,t]))_{t\ge 0} \eqd\pp{Z_\alpha(t^{1-\beta})}_{t\ge0},
\] 
but the 
random sup-measure reveals more structure than the  time-changed
extremal process. 

In this paper we fill the  gaps left in the previous studies. First of
all, we move away from the assumption of stability to a more general
class of stationary \id\ processes. This allows us to remove the
restriction of $\alpha\in (0,2)$ in our limit theorems. Much more
importantly,   we remove the assumption $\beta\in
(0,1/2)$. This allows us to consider the extrema of processes whose
memory is very long. Our 
results
confirm that the Fr\'echet limits
obtained in \cite{samorodnitsky:2004a} and the subsequent publications
disappear when $\beta\in (1/2,1)$. 
In fact, entirely new limits appear. 
 Even the one-dimensional distributions we obtain as marginal
limits have not, to the best of our knowledge, been previously
described. The limiting random sup-measure is of a shot-noise type (see e.g.~\citep{vervaat79stochastic}),
and it turns out to be  uniquely
 determined by the random upper-semi-continuous function 
\[
\eta^{\alpha,\beta}(t) := \sum_{j=1}^\infty U_j\topp\alpha\inddd{t\in{V\topp\beta_j+R\topp\beta_j}}, t\ge0,
\]
with $(U\topp\alpha_j,V\topp\beta_j,R\topp\beta_j)_{j\in\N}$ as
before. When $\beta\in(0,1/2]$, this is the same random sup-measure as
the one 
in~\eqref{eq:LS16}, as independent $\beta$-stable regenerative sets do
not intersect for such a $\beta$. For $\beta>1/2$, however, eventual
intersections occur 
almost surely, and the larger the $\beta$ becomes, more independent regenerative
sets can intersect at the same time.  
As Section~\ref{sec:RSM} below shows, for every $\alpha\in(0,\infty)$,
$(\eta^{\alpha,\beta})_{\beta\in(0,1)}$ forms a family of random
sup-measures corresponding to the full range of dependence: from
independence ($\beta\downarrow0$) to complete dependence
($\beta\uparrow1$).  
Importantly, if $\alpha\in(0,1)$, the marginal distributions, for
example those  of
$\eta^{\alpha,\beta}([0,1])$, form a family of distributions that
interpolate between the $\alpha$-Fr\'echet distribution (resulting
when $\beta \in(0,1/2]$) 
and the totally skewed to the right $\alpha$-stable distribution as $\beta\uparrow
1$.  

The paper is organized as follows. In Section~\ref{sec:prelim} we
present background information on random closed sets and  random
sup-measures. In Section~\ref{sec:RSM} we introduce and investigate
the limiting random sup-measure. The stationary \id\ process with long
range dependence whose extremes we study is introduced 
in Section~\ref{sec:id}, and a limit theorem for these extremes in the
context of random sup-measures is proved  in Section~\ref{sec:limit}. 

\section{Random closed sets and sup-measures}\label{sec:prelim}
We first provide background on  random closed sets. Our main reference
is \cite{molchanov05theory}. Let $\filF(E)$ denote the space of all
closed subsets of an interval $E\subset\R$. In this paper we only work
with $E = [0,1]$ and $E=[0,\infty)$. The space $\filF = \filF(E)$ is
equipped with the Fell topology generated by  
\[
\filF_G := \{F\in\filF:F\cap G\neq\emptyset\}\mfa G\in\calG,
\]
where $\calG = \calG(E)$  is the collection of all open subsets of $E$, and 
 \[
 \filF^K:=\{F\in\filF:F\cap K=\emptyset\} \mfa K\in\calK,
 \]
 whre $\calK = \calK(E)$ is the collection of all compact subsets of
 $E$. This topology is metrizable, and $\filF(E)$ is compact  under
 it. If $\filF$ is equipped with the Borel $\sigma$-algebra 
 $\calB(\filF)$ induced by the Fell topology, a random closed set is a
 measurable mapping from a probability space to $(\filF,
 \calB(\filF))$. Given  random closed sets $(R_n)_{n\in\N}$ and $R$, a
 sufficient condition for weak convergence $R_n\weakto R$ is 
\[
\limn\proba(R_n\cap A \neq\emptyset) = \proba(R\cap A = \emptyset), \mfa A\in\calA \cap \mathfrak S_R,
\]
 where $\calA$ is the collection of all finite unions of open intervals, and 
$\mathfrak S_R$ is the collection of all continuity sets of 
$R$: the collection of relatively compact Borel sets $B$ such that
$\proba(R\cap \wb B\neq\emptyset) = \proba(R\cap B^o\neq\emptyset)$.  
See \citet[Corollary 1.6.9]{molchanov05theory}, where the collection $\calA$ is called a {\em separating class}.

We proceed with background on  sup-measures and upper-semi-continuous
functions. Our main reference is \cite{obrien:torfs:vervaat:1990}. See
also \cite{molchanov:strokorb:2016} and \cite{sabourin17marginal}
for some recent developments. Let  $E$ be as above, and
$\calG=\calG(E)$ the collection of open subsets of $E$. A map
$m:\calG\to[0,\infty]$ is a sup-measure, if  
\[
m\pp{\bigcup_{\alpha}G_\alpha} = \sup_\alpha m(G_\alpha)
\]
for all arbitrary collections of open sets $(G_\alpha)_\alpha$. Given
a sup-measure $m$, its sup-derivative, denoted by $d^\vee
m:E\to[0,\infty]$, is defined as 
\[
d^\vee m(t):=\inf_{G\ni t}m(G), t\in E.
\]
The sup-derivative of a sup-measure is an upper-semi-continuous
function, that is a function $f$ such that $\{f<t\}$ is open for all
$t>0$. Given an $[0,\infty]$-valued  upper-semi-continuous function $f$, the
sup-integral $i^\vee f:\calG\to[0,\infty]$ is defined as 
\[
i^\vee f(G) := \sup_{t\in G}f(t), G\in\calG,
\]
with $i^\vee f(\emptyset) = 0$ by convention. The sup-integral is a
sup-measure.   
Let $\SM = \SM(E)$ and $\USC = \USC(E)$ denote the spaces of all
sup-measures on $E$ and all $[0,\infty]$-valued  upper-semi-continuous
functions on $E$, respectively.   It turns out that 
$d^\vee$ is a bijection between $\SM$ and $\USC$, and $i^\vee$ is its inverse. 
Every $m\in \SM$ has a canonical extension to all subsets of $E$,
given by 
\[
m(B) = \sup_{t\in B}(d^\vee m)(t), B\subset E.
\]

The space $\SM$ is equipped with the so-called sup-vague topology. In
this topology, $m_n\to m$ if and only if 
\[
\limsupn m_n(K)  \le m(K) \mfa K\in\calK 
\]
and
\[
\liminfn m_n(G)  \ge m(G) \mfa G\in\calG. 
\]
This topology is metrizable and the space SM is compact in this
topology. The sup-vague topology on the space USC is then induced by
the bijection $d^\vee$, so the convergence of  
\[
m_n\to m \mbox{ in SM} \qmand d^\vee m_n\to d^\vee m \mbox{ in USC}
\]
are equivalent.

A random sup-measure is a random element in $(\SM,\calB(\SM))$ with
$\calB(\SM)$ the Borel $\sigma$-algebra induced by the sup-vague
topology. A random upper-semi-continuous function is defined
similarly. We will introduce the limiting random sup-measures in our
limit theorem through their corresponding random
upper-semi-continuous functions. When proving weak convergence for
random sup-measures we will utilize the following fact: 
given random sup-measures 
$(\eta_n)_{n\in\N}$ and $\eta$, weak convergence 
$\eta_n\weakto\eta$ in $S$
 is equivalent to the finite-dimensional convergence 
\[
(\eta_n(I_1),\dots,\eta_n(I_m))\weakto(\eta(I_1),\dots,\eta(I_m))
\]
for all 
$m\in\N$ and all open and $\eta$-continuity intervals $I_1,\dots,I_m$ ($I$
is $\eta$-continuity if $\eta(I) = \eta(\wb I)$ with probability one). 
See \cite{obrien:torfs:vervaat:1990}, Theorem 3.2.
 
We will need the following result on joint convergence of random closed sets. 
\begin{Thm}\label{thm:joint_closed_sets}
Let $\{A_k\}_{k=1,\dots,m}$ and 
$\{A_k(n)\}_{n\in\N, k=1,\dots,m}$ be   random closed sets in $\filF = \filF(\Rd)$, and set 
for $I\subset\{1,\dots,m\}$ 
\[
A_I(n) = \bigcap_{k\in I}A_k(n) \qmand A_I = \bigcap_{k\in I} A_k, \mfa I\subset\{1,\dots,m\}
\]
(by convention we set $A_\emptyset(n)= A_\emptyset =  \Rd$).
Assume that 
\equh\label{eq:joint}
(A_1(n),\dots,A_m(n))\weakto (A_1,\dots,A_m)
\eque
in $\filF^m$ as $n\to\infty$.

\begin{enumerate}[(i)]
\item If $A_I = \emptyset$ almost surely, then 
$A_I(n) \weakto \emptyset$.

\item If 
\equh\label{eq:marginal_A}
A_I(n)\weakto A_I \mmas n\to\infty, \mfa I\subset\{1,\dots,m\}.
\eque
Then 
\[
\ccbb{A_I(n)}_{I\subset\{1,\dots,m\}}\weakto \ccbb{A_I}_{I\subset\{1,\dots,m\}},
\]
in $\filF^{2^m}$.
\end{enumerate}
\end{Thm}
\begin{proof}
The Fell topology on $\filF = \filF(\Rd)$ 
is compact, Hausdorff, and second countable
\citep{salinetti81convergence}. In particular, it is  metrizable and
separable. Then, by Skorokhod's  representation theorem, from \eqref{eq:joint} we can find a
common probability  space, on which $\{A_k\}_{k=1,\dots,m}$ and
$\{A_k(n)\}_{n\in\N, k=1,\dots,m}$  are defined, so that the joint law
of $\{A_k\}_{k=1,\dots,m}$ is preserved, for each $n$  the joint law
of $\{A_k(n)\}_{k=1,\dots,m}$ is preserved as well (we use the same
notation), and we have the almost sure  convergence 
\equh\label{eq:marg.as}
A_k(n)\to A_k \mmas n\to\infty, \mfa k=1,\dots,m.
\eque
Fix a nonempty set $I\subset\{1,\dots,m\}$. 
By the upper semi-continuity
in the product Fell topology of the intersection operator (see
Appendix D in \cite{molchanov05theory}),  we have
\equh\label{eq:USC}
\limsupn A_I(n)\subset A_I, \mbox{ almost surely.}
\eque
This implies the first part of the theorem immediately.

For the second part of the theorem, from now on we assume that \eqref{eq:marginal_A} holds.
We will see that \eqref{eq:marg.as} and \eqref{eq:USC} imply that 
\equh\label{eq:P}
A_I(n)\to A_I \mmas n\to\infty \mbox{ in probability.}
\eque
This will, of course, prove that $\{A_I(n)\}_{I\subset\{1,\dots,m\}}
\to \{A_I\}_{I\subset\{1,\dots,m\}}$ in probability, and, hence, also
the desired weak convergence.  

We proceed to prove \eqref{eq:P}. Recall that the distance function
for a closed set $F\in\filF(\Rd)$ is  defined as
\[
\rho(x,F) = \min\{|x-y|: y\in F\}, \quad x\in\Rd.
\]
Then, the weak convergence in the Fell topology of random closed sets
$A_I(n)\weakto A_I$, is equivalent to the weak convergence in finite-dimensional distributions of the corresponding distance functions: 
\[
\rho_n(x) := \rho(x,A_I(n))\stackrel{f.d.d.}{\to}\rho(x) =
\rho(x,A_I), \, x\in\Rd; 
\]
see \citet[Theorem 2.5]{salinetti86convergence_in_distribution}. 
Note that \eqref{eq:USC} implies that for every $x\in\Rd$, 
\[
\liminfn\rho_n(x)\ge  \rho(x), \mbox{ almost surely, }  
\]
and we conclude from 
Lemma \ref{l:liminf} below that for every $x\in\Rd$, 
\equh\label{eq:convergence_P}
\rho_n(x) \to \rho(x) \mbox{ in probability as $n\to\infty$.}
\eque
In order to establish \eqref{eq:P}, it is enough to show that for
every subsequence, there exists a further subsequence, say
$\{n_k\}_{k\in\N}$, such that $A_I(n_k)\to A_I$ almost surely as
$k\to\infty$. For this purpose we use \eqref{eq:convergence_P}  as
follows. Convergence in the Fell topology on $\Rd$ is equivalent to the
pointwise convergence of the distance functions (see
e.g.~\citep[Theorem 2.2 (iii)]{salinetti81convergence}. Suppose we show
that a further subsequence as above can be found such that 
\equh\label{eq:rationals}
\lim_{k\to\infty}\rho_{n_k}(x) = \rho(x) \mfa x \in \mathbb Q^d,
\mbox{ almost surely.} 
\eque
Since the distance functions are Lipschitz with coefficient 1, on this
event of probability 1 we actually have the entire pointwise convergence of 
the distance functions, hence the convergence in the Fell topology of
$A_I(n_k)$ to $A_I$, as required. It remains to notice that
\eqref{eq:rationals} follows from   \eqref{eq:convergence_P} by the
standard diagonalization argument. We have thus proved the theorem. 
\end{proof}
\begin{Lem} \label{l:liminf}
Let $X,\, (X_n)_{n\in\N}$ be random variables defined on the same probability
space. Suppose that $X_n \weakto X$ and $\liminfn X_n\geq X$ a.s. Then
$X_n\to X$ in probability. 
\end{Lem}
\begin{proof}
We may assume without loss of generality that all random variables
involved are uniformly bounded (e.g. by applying the arctan 
function to everything). Then $\esp X_n\to \esp X$ as $n\to\infty$. Further,
by Fatou's lemma,
$$
0\leq \limsupn \esp \bb {(X-X_n)_+} \leq \esp \bb{\limsupn
(X-X_n)_+}\leq 0\,.
$$
That is, $\esp \bb{ (X-X_n)_+}\to 0$. Then also
$$
\esp \bb{ (X-X_n)_-} =  \esp\bb{ (X-X_n)_+}  - \esp (X-X_n) \to
0\,,
$$
and so 
$$
\esp |X-X_n| = \esp \bb{ (X-X_n)_+} + \esp \bb{ (X-X_n)_-} \to  0\,.
$$
Hence $X_n\to X$ in probability. 
\end{proof} 

\section{A new family of random sup-measures}\label{sec:RSM}

Recall that for $\beta\in(0,1)$, a $\beta$-stable regenerative set is
the closure of the range of a strictly $\beta$-stable subordinator,
viewed as a random closed set in $\filF(\R_+)$, and it has Hausdorff
dimension 
$\beta$ almost surely; see for  example  \cite{bertoin:1999sub}. 
We need a result on intersections of independent stable
regenerative sets presented below. A number of similar results can be
found in literature, see for example \cite{hawkes:1977},
\cite{fitzsimmons:fristedt:maisonneuve:1095} and \cite{bertoin:1999}.
We could not however find the exact formulation needed, so we included a
short proof. 
\begin{Lem}\label{lem:intersection}
Consider $v_1,v_2\in\R_+$, $v_1\neq v_2$ and
$\beta_1,\beta_2\in(0,1)$. Let $R_1\topp{\beta_1}$ and
$R_2\topp{\beta_2}$ be two independent stable  regenerative sets with
parameter $\beta_1$ and $\beta_2$  respectively. Then, 
\begin{equation} \label{e:intersect.prob}
\proba\pp{\pp{v_1+R_1\topp{\beta_1}}\cap\pp{v_2+R_2\topp{\beta_2}}\neq\emptyset } \in\{0,1\}.
\end{equation} 
The probability equals one, if and only if
$\beta_{1,2}:=\beta_1+\beta_2-1\in(0,1)$, and in this case, the
intersection has the law of a shifted  $\beta_{1,2}$-stable
regenerative set, i.e.~a random element in $\filF(\R_+)$ with a
representation 
$$
V+R\topp{\beta_{1,2}}\,,
$$
where $R\topp{\beta_{1,2}}$ is a $\beta_{1,2}$-stable regenerative set, and $V>\max(v_1,v_2)$ is a random variable independent of $R\topp{\beta_{1,2}}$. 
\end{Lem}
\begin{proof}
We may and will assume that $v_1>v_2=0$, and drop the subscript in
$v_1$. For $x>0$ and $i=1,2$ let 
$B_{x,\beta}$ be the overshoot of the point $x$ by a strictly $\beta$-stable subordinator; 
in particular,
\[
B_{x,\beta_i} \eqd \min\pp{R_i\topp{\beta_i}\cap [x,\infty)}-x, \quad x\ge 0.
\] Define
a sequence of positive random variables $A_0,A_1,\ldots$ by
$A_0=v$, $A_{2n+1}=
B_{A_{2n},\beta_2}^{(2n+1)}$, $n=0,1,2,\ldots$,
$A_{2n}=
B_{A_{2n-1},\beta_1}^{(2n)}$, $n=1,2,\ldots$, where different
superscripts correspond to overshoots by independent
subordinators. Then, by the strong Markov property, the probability of
a nonempty intersection in \eqref{e:intersect.prob} is simply 
\begin{equation} \label{e:sum.An}
\proba\left( \sum_{n=0}^\infty A_n<\infty\right)\,.
\end{equation}
The overshoot $B_{x,\beta}$ has the density given by 
\begin{equation} \label{e:ov.density}
p_B\topp\beta(y\mid x) = \frac1{\Gamma(\beta)\Gamma(1-\beta)}\pp{\frac xy}^\beta\frac1{x+y}, y>0
\end{equation}
(see e.g.~\cite{kyprianou:2006}, Exercise 5.8.) This
implies that $B_{x,\beta}\eid xB_{1,\beta}$ for $x>0$. Grouping the terms
together, we see that the probability in \eqref{e:sum.An} is equal to 
$$
\proba\left( \sum_{n=1}^\infty \prod_{j=1}^n C_j<\infty\right)\,,
$$
where $C_1,C_2,\ldots$ are i.i.d.~random variables with $C_1\eid
B_{1,\beta_1}^{(1)}B_{1,\beta_2}^{(2)}$. An immediate conclusion is that the 
$$
\proba\left( \sum_{n=0}^\infty A_n<\infty\right) = \left\{
\begin{array}{ll}
1 & \text{if} \ \esp\log C_1<0, \\
0 & \text{if} \ \esp\log C_1\geq 0\,.
\end{array} \right.
$$
However, by \eqref{e:ov.density}, after some elementary manipulations
of the integrals, we have, writing $c(\beta) = 1/(\Gamma(\beta)\Gamma(1-\beta))$
\[
\esp\log C_1 =  c(\beta_1) \int_0^\infty \frac{y^{-\beta_1
      }\log y}{1+y}\, dy + c(\beta_2) \int_0^\infty \frac{y^{-\beta_2
      }\log y}{1+y}\, dy  = \varphi(\beta_1)-\varphi(1-\beta_2)
\]
with 
$$
\varphi(\beta) = \pp{\int_0^\infty \frac{y^{-\beta
      }\log y}{1+y}\, dy} \bigg /\pp{\int_{0}^\infty \frac{y^{-\beta}}{1+y}\, dy}.
$$
So if $\beta_2 = 1-\beta_1$, $\esp\log C_1 = 0$, and 
 it is enough to prove that the function $\varphi(\beta)$
is strictly decreasing in $\beta\in (0,1)$. To see this,
\begin{align*}
\varphi^\prime(\beta_1) & = \pp{\int\frac{y^{-\beta_1}}{1+y}\, dy}^{-2}\bb{\left(\int \frac{y^{-\beta_1
      }\log y}{1+y}\, dy\right)^2 - \int \frac{y^{-\beta_1
      }(\log y)^2}{1+y}\, dy \int\frac{y^{-\beta_1
      }}{1+y}\, dy } \\
& = -{\rm Var}\bigl(\log B_{1,\beta_1}\bigr)<0\,.
\end{align*}
This proves \eqref{e:intersect.prob}
together with the criterion for the value of 1. Finally, by the strong
Markov property of the stable regenerative sets, 
if $\beta_1+\beta_2>1$,
then
$$
\pp{v+R_1\topp{\beta_1}}\cap R_2\topp{\beta_2} \eid 
\sum_{n=0}^\infty A_n + \left( R_1\topp{\beta_1}\cap R_2\topp{\beta_2}\right)\,,
$$
where on the right hand side, the series is independent of the stable
regenerative sets. Since it has been shown by \cite{hawkes:1977} that 
$$
R_1\topp{\beta_1}\cap R_2\topp{\beta_2}\eid
R\topp{\beta_1+\beta_2-1}\,,
$$
the proof of the lemma is complete. 
\end{proof}

We now proceed with defining a new class of random sup-measures, by first identifying 
 the underlying random upper-semi-continuous function.  
From now on,  $\beta\in(0,1)$ and $\alpha>0$ are fixed parameters. 
Consider a Poisson  point process on
$\R_+\times\R_+\times \filF(\R_+)$ 
with mean measure 
\[
\alpha u^{-(1+\alpha)}\, du  (1-\beta) v^{-\beta} \, dv\, 
dP_{R\topp\beta},
\]
where $P_{R\topp\beta}$ is the law of  the $\beta$-stable
regenerative set. 
We let
$(U\topp\alpha_j,V_j\topp{\beta}, R\topp\beta_j)_{j\in\N}$ denote a
measurable enumeration of the points of the point process, and
\[
\wt R_j\topp\beta:=V_j\topp\beta+R_j\topp\beta, \, j\in\N
\]
denote the random closed sets $R_j\topp\beta$ shifted by
$V_j\topp\beta$. These are again  random closed sets. 

Introduce the
intersection of such random closed sets with indices from $S\subseteq \N$ by
\equh\label{eq:IS}
 I_S:= \bigcap_{j\in S}\wt R_j\topp{\beta}, S\neq\emptyset \qmand I_\emptyset:= \R_+.
\eque
Let 
\[
\ell_\beta:=\max\ccbb{\ell\in\N:\ell<\frac1{1-\beta}} \in \N. 
\]
By Lemma~\ref{lem:intersection}, we know that 
\equh\label{eq:IS_beta}
I_S \left\{ \begin{array}{ll}
\not=\emptyset & \text{a.s.~if} \ \ |S|\leq \ell_\beta \\
=\emptyset & \text{a.s.~if} \ \ |S|> \ell_\beta.
\end{array}\right.
\eque
Furthermore when $|S|\le \ell_\beta$,  $I_S$ is a randomly shifted stable
regenerative set with parameter $\beta_{|S|}$, where 
\[
\beta_\ell:=\ell\beta - (\ell-1) \in (0,1) \qmfa \ell=1,\dots,\ell_\beta.
\]

Let
\begin{equation} \label{e:usc.pos}
\eta^{\alpha,\beta}(t):=\sum_{j=1}^\infty
U_j\topp\alpha\inddd{t\in\wt R_j\topp\beta}, \, t\in\R_+.
\end{equation} 
Since a stable regenerative set does not hit fixed points,  for every $t$,
$\eta^{\alpha,\beta}(t) = 0$ almost surely.  Furthermore, on an
event of probability 1, every $t$ belongs to at most $\ell_\beta$
different  $\wt R_j\topp\beta$, and thus  $\eta^{\alpha,\beta}(t)$
is well-defined for all $t\in\R_+$.  In order to see that it is, on an
event of probability 1, an upper-semi-continuous function, it is
enough to prove its upper-semi-continuity on $[0,T]$ for every $T\in
(0,\infty)$. Fixing such $T$, we denote by $U\topp\alpha_{(j,T)}$ the $j$th
largest value of $U\topp\alpha_j$ for which $V_j\topp{\beta}\in [0,T]$,
$j=1,2,\ldots$. We write for
$m=1,2,\ldots$,
\begin{eqnarray*}
\eta^{\alpha,\beta}(t)& =&\sum_{j=1}^m
U\topp\alpha_{(j,T)}\inddd{t\in\wt R_j\topp\beta}
+ \sum_{j=m+1}^\infty 
U\topp\alpha_{(j,T)}\inddd{t\in\wt R_j\topp\beta} \\
&=:& \eta_{1,m}^{\alpha,\beta}(t) +\eta_{2,m}^{\alpha,\beta}(t) , \,
t\in [0,T].
\end{eqnarray*}
The random function $\eta_{1,m}^{\alpha,\beta}$ is, for every $m$,
upper-semi-continuous as a finite sum of upper-semi-continuous
functions. Furthermore, on an event of probability 1, 
$$
\sup_{t\in [0,T]} \bigl| \eta_{2,m}^{\alpha,\beta}(t)\bigr| \leq 
\sum_{j=m+1}^{m+\ell_\beta} U\topp\alpha_{(j,T)}\to 0
$$ 
as $m\to\infty$, whence the upper-semi-continuity of
$\eta^{\alpha,\beta}$. 
We define  the random sup-measure corresponding to $\eta^{\alpha,\beta}$ by 
\[
\eta^{\alpha,\beta}(G):=\sup_{t\in G}\eta^{\alpha,\beta}(t), \ G\in
\mathcal G, \ \text{the collection of open subsets of $\R_+$.}
\]
As usually, one may extend, if necessary,  the domain of $\eta^{\alpha,\beta}$ to all
 subsets of $\R_+$. We emphasize that we use the same notation $\eta^{\alpha,\beta}$ for both the random upper-semi-continuous function and the random sup-measure without causing too much ambiguity,  thanks to the homeomorphism between the spaces SM$(\R_+)$ and USC$(\R_+)$. 
 It remains to show the measurability of $\eta^{\alpha,\beta}$. Recall that the sup-vague topology of ${\rm SM} \equiv {\rm SM}(\R_+)$ has sub-bases consisting of
 \[
 \ccbb{m\in {\rm SM}: m(K)<x}, \ccbb{m\in {\rm SM}: m(G)>x} ,
 K\in\calK, G\in\calG, x\in\R_+. 
 \]
See for example \cite{vervaat:1997}, Section 3. Then, for every $x>0$, 
 \[
 \ccbb{\eta^{\alpha,\beta}(K)<x} = \bigcap_{S\subset\N}\pp{\ccbb{\sum_{j\in S}U_j\topp\alpha<x}\cap\ccbb{I_S\cap K\neq\emptyset}}
 \]
 is clearly measurable for $K\in\calK$, and so is $\{\eta^{\alpha,\beta}(G)>x\}$ for $G\in\calG$. The measurability thus follows. 
 
 
\begin{Prop} 
The  random sup-measure $\eta^{\alpha,\beta}$ is stationary and $H$-self-similar with
$H=(1-\beta)/\alpha$. 
\end{Prop}
\begin{proof}
To prove the stationarity of $\eta^{\alpha,\beta}$ as a random sup-measure it is enough to
prove that the random upper-semi-continuous function $\eta^{\alpha,\beta}$ defined
in~\eqref{e:usc.pos} has a shift-invariant law. Let $r>0$ and consider
the upper-semi-continuous function $( \eta^{\alpha,\beta}(t+r))_{t\in\R_+}$. Note
that 
\[
\eta^{\alpha,\beta}(t+r)  = \sum_{j=1}^\infty
U_j\topp\alpha\inddd{t+r\in\wt R_j\topp\beta} 
 = \sum_{j=1}^\infty
U_j\topp\alpha\inddd{t\in G_r\pp{\wt R_j\topp\beta}}, \, t\in \R_+,
\]
where $G_r$ is a map from ${\mathcal F(\R_+)}$ to ${\mathcal F(\R_+)} $, defined by 
$$
G_r(F) :=F\cap [r,\infty)-r\,.
$$
However, by Proposition 4.1 (c) in \cite{lacaux:samorodnitsky:2016}, the
map 
$$
(x,F)\to \bigl(x,G_r(F)\bigr) 
$$
on $\R_+ \times {\mathcal F(\R_+)}$ leaves the mean measure of
the Poisson random measure determined by
$(U_j\topp\alpha,   \wt R\topp\beta_j)_{j\in\N}$ 
unaffected. Hence, the law of the random upper-semi-continuous function $\bigl(\eta^{\alpha,\beta}(t+r)\bigr)_{t\in \R_+}$ coincides with that of $\bigl(\eta^{\alpha,\beta}(t)\bigr)_{t\in \R_+}$.

Similarly, in order to prove the $H$-self-similarity of
$\eta^{\alpha,\beta}$ as a random sup-measure it is enough to 
prove that the random upper-semi-continuous function $\eta^{\alpha,\beta}$ defined
in~\eqref{e:usc.pos}   is $H$-self-similar. To this end, let $a>0$, and
note that by Proposition 4.1 (b) in \cite{lacaux:samorodnitsky:2016}
\begin{align*}
\eta^{\alpha,\beta}(at)  & =  \sum_{j=1}^\infty
U_j\topp\alpha\inddd{at\in\wt R_j\topp\beta}  =  \sum_{j=1}^\infty
U_j\topp\alpha  \inddd{t\in  a^{-1}  V_j\topp\beta + a^{-1} R_j\topp\beta} \\
& \eid a^{(1-\beta)/\alpha} \sum_{j=1}^\infty
U_j\topp\alpha\inddd{t\in\wt R_j\topp\beta}
\end{align*}
jointly in $t$. 
Therefore, $\bigl( \eta^{\alpha,\beta}(at)\bigr)_{t\in \R_+}$ and  $\bigl( a^{(1-\beta)/\alpha} \eta^{\alpha,\beta}(t)\bigr)_{t\in \R_+}$ have the same law as the random upper-semi-continuous  functions.  
\end{proof}
If we restrict the random upper-semi-continuous functions and
random measures above to a compact interval, we can use a
particularly convenient  measurable enumeration of the points of the
Poisson process. Suppose, for simplicity, that the compact interval in
question is the unit interval $[0,1]$. The Poisson random
 measure $(U_j\topp\alpha,V_j\topp{\beta}, R\topp\beta_j)_{j\in\N}$ 
 restricted to $\R_+\times [0,1]\times \filF(\R_+)$ can then be
 viewed as a Poisson point process $(U_j\topp\alpha)_{j\in\N}$ on $\R_+$ with
 mean measure $\alpha u^{-(1+\alpha)}\, du$ marked by two independent
 sequences $(V_j\topp{\beta})_{j\in\N}$ and $(
 R\topp\beta_j)_{j\in\N}$  of i.i.d.~random variables. The sequence $(
 R\topp\beta_j)_{j\in\N}$ is as before, while
 $(V_j\topp{\beta})_{j\in\N}$ is a sequence of random variables
 taking values in $[0,1]$ with the common law $\proba(V_j\topp{\beta}\le v)
 = v^{1-\beta}, v\in[0,1]$.  Furthermore, we can enumerate the points of so-obtained
 Poisson random measure according to the decreasing value of the first
 coordinate, and express $(U_j\topp\alpha)_{j\in\N}$ as $(\Gamma_j^{-1/\alpha})_{j\in\N}$ with  $(\Gamma_j)_{j\in\N}$ denoting the arrival times of the
 unit rate Poisson process on $(0,\infty)$. This leads to the following  representation
\begin{equation} \label{e:usc.spec}
\pp{\eta^{\alpha,\beta}(t)}_{t\in[0,1]}\eqd\pp{\sum_{j=1}^\infty
\Gamma_j^{-1/\alpha}\inddd{t\in\wt R_j\topp\beta}}_{t\in[0,1]}.
\end{equation} 

\medskip

To conclude this section we 
would like to draw the attention of the reader to the fact that for
every fixed $\alpha\in(0,\infty)$, the family of random sup-measures
$(\eta^{\alpha,\beta})_{\beta\in(0,1)}$  interpolates certain familiar
random sup-measures. 
On one hand, as $\beta\downarrow0$, the limit is well known and
simple. To see this, notice first that for
$(U\topp\alpha_j,V\topp\beta_j,R\topp\beta_j)_{j\in\N}$  representing
the Poisson random measure on $\R_+\times\R_+\times\filF(\R_+)$ with
mean measure $\alpha
u^{-(1+\alpha)}(1-\beta)v^{-\beta}dvdP_{R\topp\beta}$, one can extend
the range of parameters to include $\beta = 0$ by setting  
$P_{R\topp0}:=\delta_{\{0\}}$ as a probability distribution (unit point mass at $\{0\}$) on
$(\filF(\R_+),\calB(\filF(\R_+))$. This is natural as
$R\topp\beta\weakto \{0\}$ in $\filF(\R_+)$ as $\beta\downarrow0$,
which follows, for example, from \cite{kyprianou:2006}, Exercise 5.8
(the ``zero-stable subordinator'' can be thought of as a process  staying an exponentially distributed amount of time at zero and then ``jumping to
infinity''.) 
It then follows that
\[
\eta^{\alpha,\beta}(\cdot)\weakto \eta^{\alpha,0}(\cdot) := \bigvee_{j=1}^\infty U\topp\alpha_j\inddd{V_j\topp0\cap\cdot\neq\emptyset}
\]
as $\beta\downarrow0$. 

The random sup-measure $\eta^{\alpha,0}$ above is the independently
scattered (a.k.a.~completely random) $\alpha$-Fr\'echet max-stable
random sup-measure on $\R_+$ with Lebesgue measure as the control
measure (see \cite{stoev:taqqu:2005} and
\cite{molchanov:strokorb:2016}). Furthermore,
$(\eta^{\alpha,\beta}([0,t]))_{t\ge0}$ corresponds to the extremal
process $(Z_\alpha(t))_{t\ge0}$ in~\eqref{eq:FELT} for a sequence of
i.i.d.~random variables with tail index $\alpha$.  The extremal
process $Z_\alpha$ also belongs to the class of $\alpha$-Fr\'echet
max-stable processes (see
e.g.~\cite{dehaan:1984},~\cite{kabluchko:2009}). 

In the range $\beta\in(0,1/2]$, the structure of $\eta^{\alpha,\beta}$
can also be simplified. As there are no intersections among independent
shifted $\beta$-stable regenerative sets,  the random sup-measure on
the positive real line becomes 
\[
 \eta^{\alpha,\beta}(\cdot) = \bigvee_{j=1}^\infty{ U\topp\alpha_j\inddd{\wt R_j\topp\beta\cap \cdot\neq \emptyset}}, \quad \beta\in(0,1/2].
\]
This random sup-measure was first studied in
\cite{lacaux:samorodnitsky:2016}. This is an $\alpha$-Fr\'echet
max-stable random sup-measure, belonging to the class
of the
so-called Choquet random sup-measures introduced in
\cite{molchanov:strokorb:2016}. It is also known that for
$\beta\in(0,1/2]$,  $(\eta^{\alpha,\beta}([0,t]))_{t\ge0}$ has the
same distribution as the time-changed extremal process
$(Z_\alpha(t^{1-\beta}))_{t\ge0}$; see
\cite{owada:samorodnitsky:2015a} and
\cite{lacaux:samorodnitsky:2016}.

On the other hand, as soon as $\beta>1/2$, the random sup-measure $\eta^{\alpha,\beta}$ is no longer an $\alpha$-Fr\'echet random sup-measure, due to the appearance of intersections.  
As $\beta\uparrow 1$, the sets $\wt R\topp\beta$ become larger and
larger in terms of Hausdorff dimension, and more and more
$U\topp\alpha_j$s enter the sums defining the random measure due to
intersections of more and more $\wt R_j\topp\beta$. In the limit, $\wt
R\topp\beta\weakto [0,\infty)$ in $\filF(\R_+)$ as $\beta\uparrow1$
(the ``one-stable subordinator'' is just the straight line).
In the limit, therefore, all $U_j\topp\alpha$s contribute to the sum 
determining the random sup-measure, but for the infinite sum to be
finite, restricting ourselves to the case
$\alpha\in(0,1)$ is necessary. In this case we have 
\[
\eta^{\alpha,\beta}(\cdot)  \weakto \eta^{\alpha,1}(\cdot) :=\pp{\sif j1 U_j\topp\alpha}\inddd{\cdot\cap\R_+\neq\emptyset}
\]
as $\beta\uparrow 1$. In words, the limit is a random sup-measure with {\em complete dependence} that takes the same value $\sif j1 U\topp\alpha_j$ on every open interval. Note that this random series follows the totally skewed $\alpha$-stable distribution.

In particular, for every $\alpha\in(0,1)$, the distributions of random
variables $(\eta^{\alpha,\beta}((0,1)))_{\beta\in[0,1]}$ interpolate
between the $\alpha$-Fr\'echet distribution ($\beta = 0$) and the
totally skewed $\alpha$-stable distribution ($\beta = 1$). These
distributions, to the best of our knowledge, have not been described
before. 
Their properties will be the subject of future
investigations. 
See \citet{simon14comparing} for a recent result on comparison between totally skewed stable and Fr\'echet distributions.

The tail behaviour of   $\eta^{\alpha,\beta}((0,1))$
is, however, clear, and it is described in the following simple
result.  
\begin{Prop}
For all $\alpha\in(0,\infty)$, $\beta\in(0,1)$,
\[
x^\alpha \proba\pp{\eta^{\alpha,\beta}((0,1))>x} \to 1
\]
as $x\to\infty$. 
\end{Prop}
\begin{proof}
Consider the representation \eqref{e:usc.spec}. 
Since $\proba(\wt R\topp \beta\cap(0,1)\neq\emptyset) = 1$,
with probability one 
\[
\Gamma_1^{-1/\alpha} \le\eta^{\alpha,\beta}((0,1))\le {\Gamma_1^{-1/\alpha}+(\ell_\beta-1)\Gamma_2^{-1/\alpha}}.
\]
Note that $\proba(\Gamma_1^{-1/\alpha}>x) \sim x^{-\alpha}$ as
$x\to\infty$, and that for $\delta\in(0,\alpha)$,  
\[
\proba\pp{\Gamma_2^{-1/\alpha}>x}\le \frac{\esp\Gamma_2^{-(\alpha+\delta)/\alpha}}{x^{\alpha+\delta}}= \frac{\Gamma(1-\delta/\alpha)}{x^{\alpha+\delta}}, \quad x>0,
\]
where $\Gamma(x)$ is the Gamma function. Hence the result. 
\end{proof}

As we shall see below,  for each $\alpha,\beta$ the random sup-measure
$\eta^{\alpha,\beta}$ arises in the limit of the extremes of stationary
processes: while $\alpha$ indicates the tail behavior, $\beta$
indicates the length of memory. The limiting case $\beta=0$
corresponds to the  short memory case already extensively
investigated in the literature, and the case $\beta\in(0,1)$
corresponds to the long range dependence regime. The larger the
$\beta$ is, the longer the memory becomes.

\section{A family of stationary \id\  processes}\label{sec:id}
We consider a discrete-time stationary symmetric \id\ process whose function
space L\'evy measure is based
on an underlying null-recurrent Markov chain. Similar models have been
investigated in  the symmetric $\alpha$-stable (S$\alpha$S)  case
in~\cite{resnick:samorodnitsky:xue:2000},
\cite{samorodnitsky:2004a}
\cite{owada:samorodnitsky:2015,owada:samorodnitsky:2015a},
\cite{owada:2016} and \cite{lacaux:samorodnitsky:2016}, which can be
consulted for various background facts stated below. 
We first describe the Markov chain. 
Consider an irreducible aperiodic
null-recurrent Markov chain $(Y_n)_{n\in\N_0}$ on $\Z$ with $\N_0 =
\{0\}\cup\N$.  Fix a state $i_0$, and let $(\pi_i)_{i\in\Z}$ be the
unique invariant measure on $\Z$ such that $\pi_{i_0} = 1$.  
Consider the space $(E,\calE) = (\Z^{\N_0},\calB(\Z^{\N_0}))$. We denote each element of $E$ by $x \equiv (x_0,x_1,\dots)$. Let $P_i$ denote the probability measure on $(E,\calE)$ determined by the Markov chain starting at $Y_0 = i$, and introduce an infinite $\sigma$-finite measure on $(E,\calE)$ defined by
\[
\mu(B):= \sum_{i\in\Z}\pi_iP_i(B), \, B\in\calE.
\]
Consider 
\[
A_0:= \{x\in E: x_0 = i_0\},
\] and the first entrance time of $A_0$
\[
\varphi_{A_0}(x):=\inf\{n\in\N: x_n = i_0\}, \, x\in E.
\]
The key assumption is that, for some $\beta\in(0,1)$ and a slowly
varying function $L$, 
\equh\label{eq:F}
\wb F(n)\equiv P_{i_0}(\varphi_{A_0}>n) = n^{-\beta}L(n)\,.
\eque 
This assumption can also be expressed
in terms of the so-called {\em wandering rate sequence} defined by 
\[
w_n:=\mu\pp{\bigcup_{k=0}^{n-1}\ccbb{x\in E: x_k = i_0}}, n\in\N.
\]
Then 
\[
w_n\sim \mu(\varphi_{A_0}\le n) \sim \summ k1n
P_{i_0}(\varphi_{A_0}\ge k), 
\]
and the key assumption  becomes $w_n\in
RV_{1-\beta}$. Here and in the sequel, 
$RV_{-\alpha}$ stands for the family of functions on $\N_0$ that are 
regularly varying at
infinity with index $-\alpha$. For technical reasons we will assume,
additionally, that 
with $p_n := \proba(Y_1 = n)$, 
\equh\label{eq:Doney}
\sup_{n\in\N}\frac{np_n}{\wb F(n)}<\infty.
\eque

If 
$T$ denotes the shift operator $T(x_0,x_1,\dots) =
(x_1,x_2,\dots)$, then $\mu$ is $T$-invariant: $\mu(\cdot) = \mu(
T^{-1}\cdot)$ on $(E,\calE)$. Furthermore,  $T$ is conservative and ergodic
with respect to $\mu$ on $(E,\calE)$.  
Next, we shall consider non-negative functions from $L^\infty(\mu)$
supported by ${A_0}$. Fix $\alpha>0$. For a fixed $f\in L^\infty(\mu)$, write 
\begin{equation} \label{e:b.n}
b_n:=\pp{\int\max_{k=0,\dots,n}\pp{f\circ T^k(x)}^\alpha\mu(dx)}^{1/\alpha}, n\in\N.
\end{equation}
The sequence $(b_n)$ satisfies 
\begin{equation} \label{e:bn.wn}
\limn\frac{b_n^\alpha}{w_n} = \nn f_\infty. 
\end{equation}
Given a Markov chain as above and $f\in L^\infty(\mu)$ supported by ${A_0}$,
we define a stationary symmetric \id\ process as a 
stochastic integral 
\begin{equation} \label{e:the.process}
X_n:= \int_E f_n(x)M(dx) \qmwith f_n := f\circ T^n, \, n\in\N_0,
\end{equation} 
where $M$ is a homogeneous symmetric \id\ random measure on $(E,\calE)$
with control measure  $\mu$ and a local L\'evy measure $\rho$, symmetric and 
satisfying 
\begin{equation} \label{e:local.power}
\rho \bigl( (z,\infty)\bigr) = az^{-\alpha} \ \ \text{for $z\geq
  z_0>0$.}
\end{equation}
We refer the reader to Chapter 3 in \cite{samorodnitsky:2016} for more
details on integrals with respect to \id\ random measures and, in
particular, for the fact that the stochastic process in
\eqref{e:the.process}  is a well-defined stationary \id\ process (Theorem 3.6.6 therein). In particular, this
process satisfies 
$$
\proba(X_0>x)\sim { a\| f\|_\alpha^\alpha \, x^{-\alpha}}
$$
as $x\to\infty$; see \cite{rosinski:samorodnitsky:1993}. We will use the value of
$\alpha$ defined by \eqref{e:local.power}   in \eqref{e:b.n}. Below we
will  work with a more explicit and helpful series
representation,~\eqref{eq:series}  of the processes of interest.

We would like to draw the attention of the reader to the fact that we
are assuming in \eqref{e:local.power} that the tail of the
local L\'evy measure has, after a certain point, exact power-law
behavior. This is done purely for clarity of the presentation. There
is no doubt whatsoever that limiting results similar to the one we
prove in the next section hold under a much more general assumption of
the regular variation of the tail of $\rho$. However, the analysis in
this case will involve additional layers of approximation that might
obscure the nature of the new limiting process we will obtain (note,
however, that the assumption \eqref{e:local.power} already covers the
S$\alpha$S case when $\alpha\in(0,2)$.)
In a similar vein, for the sake of clarity, we will assume in the next
section that $f$ is simply the indicator function of the set ${A_0}$.

Other types of limit theorems for this and related class of processes
have been investigated for the partial sums (by
\cite{owada:samorodnitsky:2015,jung17functional}) and for
the sample covariance functions (by
\cite{resnick:samorodnitsky:xue:2000,owada:2016}). In all cases
non-standard normalizations, or even new limit processes, show up in
the limit theorems, indicating long range dependence
in the model.  Properties of stationary infinitely divisble
processes have  intrinsic connections to infinite ergodic theory
(see
\cite{rosinski:1995,samorodnitsky:2005,kabluchko:stoev:2016}), and the
family of processes we are considering are said to be {\em
  driven by a null-recurrent flow}. The mixing properties of such
processes  (in the
S$\alpha$S case with $\alpha\in(0,2)$) were investigated in
\cite{rosinski:samorodnitsky:1996}. 

\section{A limit theorem for stationary \id\ processes}\label{sec:limit}
Consider the  stationary \id\ process  introduced in
\eqref{e:the.process}. For $n=1,2,\ldots $ we define a   random
sup-measure by 
\[
M_n(B):= \max_{k\in nB} X_k, \, B\subset[0,\infty)\,.
\]
The main result of this paper is the following theorem. 
\begin{Thm}\label{thm:SM}
Consider the stationary infinitely divisible process
$(X_n)_{n\in\N_0}$ defined in the previous section. Let $f={\bf
  1}_{A_0}$ with ${A_0} = \{x\in E:x_0 = i_0\}$.   Under the
assumptions 
\eqref{eq:F} and \eqref{eq:Doney} and with $b_n$ as in \eqref{e:b.n}, 
 \[
\frac1{b_n} M_n \weakto  a^{1/\alpha} \eta^{\alpha,\beta}
 \] as $n\to\infty$ in the space  $\SM(\R_+)$\,, 
where  $a$ is as in \eqref{e:local.power}.
 \end{Thm}

We start with some preparation. 
Note that by \eqref{e:bn.wn}, $b_n^\alpha\in RV_{1-\beta}$. By
stationarity  it suffices to prove convergence in the space 
$\SM([0,1])$. We start by decomposing the process $(X_n)_{n\in\N_0}$ into the sum
of two independent stationary symmetric infinitely divisible
processes:
$$
X_n= X_n^{(1)} + X_n^{(2)}, \, n\in\N_0,
$$
with
$$
X_n^{(i)}:= \int_E f_n(x)M^{(i)}(dx), \,  n\in\N_0, \, i=1,2,
$$
with $f_n$ as in \eqref{e:the.process}, and $M^{(1)}$ and $M^{(2)}$
two independent homogeneous symmetric \id\ random measures on
$(E,\calE)$, each with control measure  $\mu$. The local L\'evy measure for
$M^{(1)}$ is the  measure $\rho$ restricted to the set $\{|z|\geq z_0\}$, 
while the local L\'evy measure for
$M^{(2)}$ is the  measure $\rho$ restricted to the set  $\{|z|< z_0\}$.
The first observation is that random variables $(X_n^{(2)})_{n\in\N_0}$ have
L\'evy measures supported by a bounded set, hence they have
exponentially fast decaying tails; see for example \cite{sato:1999}. Therefore, 
$$
\frac{1}{b_n} \max_{k=0,1,\ldots, n} |X_k^{(2)}|\to 0
$$
in probability as $n\to\infty$.  Therefore, without loss of generality
we may assume that, in addition to~\eqref{e:local.power}, the local
L\'evy measure $\rho$ is, to start with, 
supported by the set $\{|z|\geq  z_0\}$.

For each $n\in\N$, the random vector $(X_0,\dots,X_n)$ 
admits a
series representation that we will now describe. For $x>0$ let 
\begin{equation} \label{e:G.explicit}
G(x) :=\left\{ \begin{array}{ll}
a^{1/\alpha}x^{-1/\alpha} & 0<x< az_0^{-\alpha} \\
0 & x\ge az_0^{-\alpha}.
\end{array}
\right.
\end{equation}
It follows from Theorem 3.4.3 in~\cite{samorodnitsky:2016} that the
following representation in law holds: 
\equh\label{eq:series}
(X_k)_{k=0,\dots,n}\eqd\pp{ \sif j1
  \varepsilon_jG\bigl( \Gamma_j /2b_n^\alpha\bigr)
\inddd{T^k(U_j\topp n)_0=i_0}}_{k=0,\dots,n},
\eque
where $(\Gamma_j)_{j\in\N}$ are as in \eqref{e:usc.spec}, $(
\varepsilon_j)_{j\in\N}$ are i.i.d.~Rademacher random variables 
and 
$(U_j\topp n)_{j\in\N}$ are i.i.d.~$E$-valued random variables with
common law $\mu_n$,  determined by
\[
\frac{d\mu_n}{d\mu}(x) = \frac{1}{b_n^\alpha} \inddd{
T^k(x)_0=i_0 \ \ \text{for some $k=0,1,\ldots, n$}}, x\in E.
\]
All three sequences are independent.
Here and in the sequel, for $x\in E\equiv \Z^{\N_0}$ we write $T^k(x)_0 \equiv [T^k(x)]_0\in\Z$.
 

Our argument consists of coupling the series representation of
$\eta^{\alpha,\beta}$ in~\eqref{e:usc.spec} with the series
representation of the process in~\eqref{eq:series}. Note that the point process
$\bigl(\Gamma_j^{-1/\alpha}\bigr)_{j\in\N,\varepsilon_j = 1}$ is a Poisson
point process with mean measure $2^{-1}\alpha u^{-(1+\alpha)}\, du$,
$u>0$, and it can be represented in law as the point process
$\bigl(2^{-1/\alpha}\Gamma_j^{-1/\alpha}\bigr)_{j\in\N}$. Therefore,
we may and will work with a version of $\eta^{\alpha,\beta}$ given by
\[
\pp{\eta^{\alpha,\beta}(t)}_{t\in[0,1]}=\pp{2^{1/\alpha}\sum_{j=1}^\infty
\inddd{\varepsilon_j=1}\Gamma_j^{-1/\alpha}\inddd{t\in\wt R_j\topp\beta}}_{t\in[0,1]}.
\]
We proceed 
through a truncation argument. Introduce for $\ell=1,2,\ldots$ 
\[
M_{\ell,n}(B):=  \max_{k\in  nB}  \summ  j1\ell\varepsilon_j
 G\bigl( \Gamma_j /2b_n^\alpha\bigr)
\inddd{T^k(U_j\topp n)_0=i_0}, n\in\N,
\]
and
\equh\label{eq:alt1}
\eta_\ell^{\alpha,\beta}(t):=2^{1/\alpha}\sum_{j=1}^\ell
\inddd{\varepsilon_j=1} 
\Gamma_j^{-1/\alpha} \inddd{t\in\wt R_j\topp\beta}, \, t\in[0,1].
\eque
We also let $\eta^{\alpha,\beta}_{\ell}$ denote the corresponding truncated random sup-measure. 
The key steps of the proof of Theorem~\ref{thm:SM} are  
Propositions \ref{prop:SM_ell} and  \ref{prop:SM_remainder} below.

\begin{Prop}\label{prop:SM_ell}
Under the assumptions of Theorem~\ref{thm:SM}, for all $\ell\in\N$,
\[
\frac1{b_n}M_{\ell,n} \weakto a^{1/\alpha} \eta_\ell^{\alpha,\beta}
\]
as $n\to\infty$ in the space of $\SM([0,1])$. 
\end{Prop}

\begin{Prop}\label{prop:SM_remainder}
Under the assumptions of Theorem~\ref{thm:SM}, for all $\delta>0$, 
\[
\lim_{\ell\to\infty}\limsupn\proba\pp{\max_{k=0,\dots,n}\frac{1}{b_n}\abs{\sif
    j{\ell+1}\varepsilon_j
G\bigl( \Gamma_j /2b_n^\alpha\bigr)
\inddd{T^k(U_j\topp n)_0=i_0}
}>\delta} = 0.
\]
\end{Prop}

We start with several preliminary results needed for the proof of 
Proposition~\ref{prop:SM_ell}. First of all we establish
convergence of simultaneous return times of independent Markov chains. Introduce
\[
\what R_{j,n}\topp\beta:=\frac1n\ccbb{k\in\{0,\dots,n\}:T^k(U_j\topp
  n)_0 = i_0},
\]
\[
\what I_{S,n}:= \bigcap_{j\in S}\what
R_{j,n}\topp\beta,  S\subset\N,
S\neq\emptyset \qmand \what I_{\emptyset,n}:=\frac1n\ccbb{0,1,\dots,n}.\]
Recall the definition of $I_S$ in~\eqref{eq:IS}. 

\begin{Thm}\label{thm:joint}
Assume that \eqref{eq:F} and \eqref{eq:Doney} hold. Then for all $\ell\in\N$, 
\[
\pp{\what I_{S,n}}_{S\subset\{1,\dots,\ell\}}\weakto \pp{I_S\cap[0,1]}_{S\subset\{1,\dots,\ell\}}
\]
as $n\to\infty$ in $\filF([0,1])^{2^\ell}$, where for each $n$ the law in the left-hand side is computed  under $\mu_n$.
\end{Thm}
\begin{proof}
By the second part of Theorem \ref{thm:joint_closed_sets}, it suffices to
show the marginal convergence for $S = \{1,\dots,\ell\}$, for all $\ell\in\N$.
First, we have seen in \eqref{eq:IS_beta} that if 
\equh\label{eq:l*}
\beta^*_\ell:=\ell\beta-\ell+1\in(0,1)
\eque
is {\em violated}, then the desired limit is the deterministic empty set. The convergence then follows from the first part of Theorem \ref{thm:joint_closed_sets}.
From now on we assume \eqref{eq:l*}. For the simultaneous return times $\what I_{S,n}$ of independent Markov chains indexed by $S$, by introducing
\begin{align*}
V_{S,n}\topp\beta & := \frac1n\min\ccbb{k=0,\dots,n:T^k(U_j\topp n)_0 = i_0, \mfa j\in S}\\
R_{S,n}\topp\beta& := \frac1n\ccbb{k=0,\dots,n:T^{nV_{I,n}\topp\beta + k}(U_j\topp n)_0 = i_0, \mfa j\in S},
\end{align*}
we have the decomposition $\what I_{S,n} = V_{S,n}\topp\beta + R_{S,n}\topp\beta$. Applying 
Corollary \ref{coro:1} to $I_S$ in \eqref{eq:IS} we have that $I_S = \wt V\topp\ell + R\topp{\beta_\ell^*}$ where $\wt V\topp\ell$ satisfies \eqref{eq:Vt} with $\beta_j = \beta, j=1,\dots,\ell$ and $R\topp{\beta_\ell^*}$ is a $\beta_\ell^*$-stable regenerative set, and the two are independent. 
In summary, the desired convergence now becomes
\equh\label{eq:VR}
\pp{V_{S,n}\topp\beta + R_{S,n}\topp\beta}\cap[0,1] \weakto \pp{\wt V\topp\ell+R\topp{\beta^*_\ell}}\cap[0,1],
\eque

We first show the convergence of $V_{S,n}^{(\beta)}$ to $\wt
V\topp\ell$, and we start with the last visit decomposition 
\[
\mu_n\pp{V_{S,n}\topp\beta\le x} = 
\summ k1{\floor{nx}}\mu_n\pp{\mbox{last simultaneous return to $i_0$
    before $\floor{nx}$ is at time $k$}}.
\]
Denoting by $\wb{F^*}$  the tail of the time between two successive
simultaneous visits to $i_0$ by $\ell$ i.i.d.~Markov chains, we have
\begin{align*}
\mu_n\pp{V_{S,n}\topp\beta\le x} & = \summ k1{\floor{nx}}
                                   \mu_n\pp{T^k(U_j\topp n)_0 = i_0 
                                   \ \text{for all} \  j\in S}\wb{F^*}(\floor{nt}-k) \\
& = \frac1{w_n^{\ell}}\summ k1{\floor{nx}}\wb{F^*}(\floor{nx}-k) = \frac1{w_n^\ell}\summ k0{\floor{nx}-1}\wb{F^*}(k).
\end{align*}
By Lemma \ref{lem:1} in Appendix and Karamata's theorem,
\begin{align*}
\mu_n\pp{V_{S,n}\topp\beta\le x} & \sim \frac{(1-\beta)^\ell}{L(n)^\ell}n^{\beta^*_\ell-1} L^*(\floor{nx})\frac{\floor{nx}^{1-\beta^*_\ell}}{1-\beta^*_\ell}\\
&\to
  x^{1-\beta^*_\ell}\frac{(\Gamma(\beta)\Gamma(2-\beta))^\ell}{\Gamma(\beta^*_\ell)\Gamma(2-\beta^*_\ell)} = \proba\pp{\wt V\topp\ell\le x}
\end{align*}
as $n\to\infty$ for $x\in[0,1]$ (comparing with \eqref{eq:Vt}). 

Furthermore, the law of $nR\topp\beta_{S,n}$ is that of a renewal
process with inter-renewal times distributed as $F^*$; see Appendix \ref{sec:renewal}. Therefore, by
\citet[Theorem A.8]{giacomin07random}, $R\topp\beta_{S,n}\weakto
R\topp{\beta^*}$ in $\filF(\R_+)$ as $n\to\infty$. 
The claim \eqref{eq:VR}  now follows by an application of the
continuous mapping theorem: the map 
\[
\R_+\times \filF([0,1]) \ni (x,F)\mapsto (x+F)\cap[0,1]\in\filF([0,1])
\]
 is continuous, except at the point $\{(x,F): (x+F)\cap[0,1] = \{1\}\}$
 (e.g.~\citet[Appendix B]{molchanov05theory}). The probability that
 the latter point is  hit by $\wt
 V\topp{\ell}+R\topp{\beta^*_\ell}\cap[0,1]$ is, however, equal to
 zero. This proof is thus complete.
\end{proof}
 
Next, we show that for each open interval $T$, outside an event $A_n(T)$ to be defined below, of which the probability tends to zero as $n\to\infty$, the following key identity holds:
\begin{multline}\label{eq:Anc}
\max_{k\in nT}\summ j1\ell\varepsilon_j G\bigl( \Gamma_j /2b_n^\alpha\bigr)
\inddd{T^k(U_j\topp n)_0 = i_0}
 = 
\max_{S\subset\{1,\dots,\ell\}}\inddd{\what I_{S,n}\cap
  T\neq\emptyset}\sum_{j\in S}\varepsilon_j G\bigl( \Gamma_j /2b_n^\alpha\bigr)\\
= \max_{S\subset\{1,\dots,\ell\}}\inddd{\what I_{S,n}\cap
  T\neq\emptyset}\sum_{j\in S}\inddd{\varepsilon_j = 1}G\bigl(
  \Gamma_j /2b_n^\alpha\bigr),
   \end{multline}
with the convention that $\sum_{j\in\emptyset} = 0$. 

To establish this,  we take a closer look at the simultaneous returns of Markov chain to $i_0$. We say that the chain indexed by $j$ returns to $i_0$ at time $k$, if $T^k(U_j\topp n)_0 = i_0$.  Note that if  
\[
\frac kn\in \what I_{S,n}\cap T = \pp{\bigcap_{j\in S}\what R_{j,n}\topp\beta}\cap T,
\]
then there might be another $j'\in\{1,\dots,\ell\}\setminus S$, such that the chain indexed by $j'$ returns to $i_0$ at the same time $k$ as well. 
We need an exact description of simultaneous returns of multiple chains. For this purpose, introduce
\[
\what I^*_{S,n}:= \what I_{S,n} \cap \pp{\bigcup_{j\in\{1,\dots,\ell\}\setminus S}\what R_{j,n}\topp\beta}^c,
\]
 the collection of all time points (divided by $n$) at which all chains indexed by $S$, and only these chains, return to $i_0$ simultaneously. 
  We define  the event
\equh\label{eq:An}
A_n(T):= \bigcup_{S\subset\{1,\dots,\ell\}}\pp{\ccbb{\what I_{S,n}\cap T\neq \emptyset}\cap \ccbb{\what I^*_{S,n}\cap T = \emptyset}}.
\eque
In words, on the complement of $A_n(T)$, if $\what I_{S,n}\cap T\neq \emptyset$ for some non-empty set $S$, then at some time point $k\in nT$, exactly those chains indexed by $S$ return to $i_0$. 
\begin{Lem} \label{l:intersection.subset}
For every open interval $T$, the identity~\eqref{eq:Anc} holds on $A_n(T)^c$, and $\limn\proba(A_n(T)) = 0$.  
\end{Lem}
\begin{proof}
We first prove the first part of the lemma. 
Noticing that $S = \emptyset$ is also included in the union above, and
\[
\what I_{\emptyset,n}^*=\pp{\bigcup_{j=1,\dots,\ell}\what R_{j,n}\topp\beta}^c,
\]we see that   $A_n(T)$ includes the event that at every time $k$ at
least one of the $\ell$ chains returns to $i_0$. So on $A_n(T)^c$, the
first two terms in~\eqref{eq:Anc}, which are, clearly, always equal, 
are non-negative.  Furthermore, when $\what I_{S,n}\cap
T\neq\emptyset$ for some non-empty $S$, then for $S' :=\{j\in
S:\varepsilon_j=1\}\subset S$, $\what I_{S,n}\cap T\neq\emptyset$
implies $\what I_{S',n}\cap T\neq\emptyset$, and therefore restricted
to the event $A_n(T)^c$ we have $\what I^*_{S',n}\cap T
\neq\emptyset$. It follows that the second equality in~\eqref{eq:Anc}
also holds on $A_n(T)^c$.  

For the second part of the lemma, in view of~\eqref{eq:An}, it suffices to show for all $S$,
\[
\limn\proba\pp{\ccbb{\what I_{S,n}\cap T\neq\emptyset}\cap \ccbb{\what I^*_{S,n}\cap T=\emptyset}} = 0.
\]
The case $S = \emptyset$ is trivial. So without loss of generality, assume $S = \{1,\dots,\ell'\}$ for some $\ell'\in\{1,2,\dots,\ell-1\}$. Introduce
\[
K_n:=n\min\pp{\what I_{S,n}\cap T},
\]
the first time in $nT$ that all chains indexed by $S$ return to $i_0$ simultaneously. Then, 
\[
\ccbb{\what I_{S,n}\cap T\neq\emptyset} \cap \ccbb{\what I^*_{S,n}\cap T = \emptyset}\subset\bigcup_{j=\ell'+1}^\ell\ccbb{T^{K_n}\spp{U_{j}\topp n}_0 = i_0, \what I_{S,n}\cap T \neq \emptyset}.
\]
The probability of each event in the union on the right hand side is bounded from above by
\[
\proba\pp{T^{K_n}\spp{U\topp n_{j}}_0 = i_0\mmid \what I_{S,n}\cap T
  \neq\emptyset}\leq \max_{k=0,\dots,n}\proba\pp{T^k(U_1\topp n)_0 =
  i_0} = b_n^{-\alpha}
\]
by the i.i.d.~assumption on the chains. Since $b_n\to\infty$, the
proof is complete. 
\end{proof}
Now we are ready to prove the main result.
\begin{proof}[Proof of Proposition~\ref{prop:SM_ell}]
  By  Theorem 3.2 in \cite{obrien:torfs:vervaat:1990} 
 and the fact that the stable regenerative sets do not hit
 points, it suffices to show, for all $m\in\N$ and all disjoint open
 intervals $T_i = (t_i,t_i')\subset[0,1], i=1,\dots,m$, 
\equh\label{eq:Mln_fdd}
  \pp{\frac1{b_n}M_{\ell,n}(T_i)}_{i=1,\dots,m} \weakto
  \pp{ a^{1/\alpha}\eta^{\alpha,\beta}_\ell(T_i)}_{i=1,\dots,m} 
\eque
  as random vectors in $\R^m$.  The expression \eqref{e:G.explicit}
  and the fact that $b_n\to\infty$ tell us that the event
  $B_n:=\{\Gamma_\ell/2b_n^{\alpha}<az_0^{-\alpha}\}$ has probability
  going to 1 as $n\to\infty$, and on $B_n$ we have
  $G(\Gamma_j/2b_n^\alpha) =
  (2a)^{1/\alpha}\Gamma_j^{-1/\alpha}b_n$. 
  In  particular, on the event $B_n$ we have
  \[
  \frac1{b_n}M_{\ell,n}(T_i) = \max_{k\in nT_i}\summ j1\ell\varepsilon_j(2a)^{1/\alpha}\Gamma_j^{-1/\alpha}\inddd{T^k(U_j\topp n)_0 = i_0}.
  \] Therefore, proving~\eqref{eq:Mln_fdd} is
  the same as proving that
 \equh\label{eq:fdd}
\left( \max_{k\in   nT_i}   \summ
  j1\ell\varepsilon_j  2^{1/\alpha}\Gamma_j^{-1/\alpha}
\inddd{T^k(U_j\topp n)_0=i_0}
\right)_{i=1,\dots,m} \weakto
\pp{  \eta^{\alpha,\beta}_\ell(T_i)}_{i=1,\dots,m} \,.
 \eque
The first part of Lemma~\ref{l:intersection.subset} yields that on $A_n(T_i)^c\cap B_n$,
\[
\max_{k\in   nT_i}   \summ
  j1\ell\varepsilon_j\Gamma_j^{-1/\alpha} 
\inddd{T^k(U_j\topp n)_0 = i_0}
 = \max_{S\subset\{1,\dots,\ell\}}\inddd{\what I_{S,n}\cap
  T_i\not=\emptyset} \sum_{j\in S}
  \inddd{\varepsilon_j=1}\Gamma_j^{-1/\alpha}. 
\]
Since by Lemma
\ref{l:intersection.subset}, $\proba(A_n(T_i)^c\cap B_n)\to 1$ as $n\to\infty$, 
the statement \eqref{eq:fdd} will follow once we prove that
$$
\left( \max_{S\subset\{1,\dots,\ell\}}\inddd{\what I_{S,n}\cap  T_i\not=\emptyset} \sum_{j\in S} 2^{1/\alpha}\inddd{\varepsilon_j = 1}\Gamma_j^{-1/\alpha}\right)_{i=1,\ldots, m} \weakto \pp{ \eta^{\alpha,\beta}_\ell(T_i)}_{i=1,\dots,m} \,.
$$
This is, however, an immediate consequence of Theorem \ref{thm:joint} and the fact that
$\eta^{\alpha,\beta}_\ell(T_i)$ can be written in the form (recalling \eqref{eq:alt1})
$$
\eta^{\alpha,\beta}_\ell(T_i) =
\max_{S\subset\{1,\dots,\ell\}}\inddd{  I_{S}\cap   T_i\not=\emptyset} 2^{1/\alpha}\sum_{j\in S}\inddd{\varepsilon_j = 1}\Gamma_j^{-1/\alpha}, \ i=1,\ldots, m\,.
$$
\end{proof}
\begin{proof}[Proof of Proposition~\ref{prop:SM_remainder}]
For $M>0$ let $D^M_\ell:=\{ \Gamma_{\ell+1}\geq M \}$. It is clear
that $\lim_{\ell\to\infty}\proba(D^M_\ell)= 1$.  We have 
\begin{multline*}
 \proba \pp{\left\{ \max_{k=0,\dots,n}\frac{1}{b_n}\left|\sif
    j{\ell+1}\varepsilon_j
G\bigl( \Gamma_j /2b_n^\alpha\bigr)
\inddd{T^k(U_j\topp n)_0=i_0}\right|
>\delta \right\}  \cap D_\ell^M} \\
 \leq  \sum_{k=0}^n \proba\left( \left\{\left|\sum_{j=\ell+1}^\infty \varepsilon_j
       \Gamma_j^{-1/\alpha} \inddd{ \Gamma_j\leq 2ab_n^\alpha
       z_0^{-\alpha}}\inddd{ T^k(U_j\topp
       n)_0=i_0}\right|>\frac\delta{(2a)^{1/\alpha}}  \right\} \cap
       D_\ell^M\right).
\end{multline*}
Note that on the right hand side above, the summand takes the same value for all $k=0,1,\dots,n$. Write $\delta' := \delta/(2a)^{1/\alpha}$. We shall show that, for all $\delta'>0$, one can choose $M$ depending on $\alpha$, $\beta$ and $\delta'$ only, such that  for all $\ell$,
\[
\limsupn n\proba\pp{\ccbb{\abs{\sum_{j=\ell+1}^\infty\varepsilon_j\Gamma_j^{-1/\alpha}\inddd{\Gamma_j\le 2ab_n^{\alpha}z_0^{-\alpha}}\inddd{T^k(U_j\topp n)_0 = i_0}}>\delta'}\cap D_\ell^M} = 0.
\]
The desired result then follows. 
To show the above, first observe that the probability of interest is bounded from above by
\equh\label{eq:M}
\proba\left(  \sum_{j=1}^\infty 
       \Gamma_j^{-1/\alpha} \inddd{M\leq \Gamma_j\leq 2ab_n^\alpha
       z_0^{-\alpha}}\inddd{T^k(U_j\topp
       n)_0=i_0}>\delta' \right).
\eque
Observe that  the restriction
to $(0,\infty)$ of the point process with the points
$$
\pp{b_n\Gamma_j^{-1/\alpha} \inddd{T^k(U_j\topp
       n)_0=i_0}}_{j\in\N}
$$
represents a Poisson random measure on $(0,\infty)$ with intensity
$\mu(A_0)\alpha u^{-(\alpha+1)}\, du$, $u>0$, and another representation
of the same Poisson random measure is 
$$
 \pp{ \mu(A_0)^{1/\alpha}\Gamma_j^{-1/\alpha}}_{j\in\N}.
$$
By definition of the Markov chain, $\mu(A_0) = 1$. 
Therefore,~\eqref{eq:M}  becomes
\begin{multline}
  \proba\left( b_n\inv
\sum_{j=1}^\infty \Gamma_j^{-1/\alpha}  
 \inddd{ M/b_n^\alpha\leq \Gamma_j
 \leq 2a
       z_0^{-\alpha}} >\delta'  \right)
       \\
        \le  \proba\left( b_n\inv
\sum_{j=j_M+1}^\infty  \Gamma_j^{-1/\alpha}  
\inddd{  \Gamma_j\leq 2a 
       z_0^{-\alpha}} >\delta'/2\right)\label{eq:M2}
 \end{multline}
by taking $j_M := \sfloor{M^{1/\alpha}\delta'/2}$,
 so that 
$b_n\inv\summ j1{j_M} \Gamma_j^{-1/\alpha}\inddd{M/b_n^\alpha\le \Gamma_j\le 2az_0^{-\alpha}} \le  \delta'/2$ with probability one.
By Markov inequality, we can further bound~\eqref{eq:M2} by, up to a multiplicative constant depending on $\delta'$,
\[
b_n^{-p}  \esp  \left( \sum_{j=j_M+1}^\infty \Gamma_j^{-1/\alpha} 
\inddd{  \Gamma_j\leq 2a 
       z_0^{-\alpha}}\right)^{p}.
\]
If we choose
 $p>1/(1-\beta)$,
then  $b_n^{-p} = o(n\inv)$. Since choosing $M$ and, hence, $j_M$
large enough, we can ensure finiteness of the above expectation, this
completes the proof. 
\end{proof}
\begin{proof}[Proof of Theorem~\ref{thm:SM}]
As in the proof of Proposition~\ref{prop:SM_ell}, it suffices to show,
for all $m\in\N$ and all disjoint open 
 intervals $T_i = (t_i,t_i')\subset[0,1], i=1,\dots,m$,
\[
\left( \max_{k\in   nT_i}  \frac{1}{b_n}\summ
  j1\infty\varepsilon_j G\bigl( \Gamma_j /2b_n^\alpha\bigr)
\inddd{ T^k(U_j\topp n)_0=i_0}\right)_{i=1,\dots,m} \weakto
\pp{ a^{1/\alpha}\eta^{\alpha,\beta}(T_i)}_{i=1,\dots,m} \,.
\]
We will use Theorem 3.2 in \cite{billingsley:1999}. By
Proposition~\ref{prop:SM_ell} and the obvious fact that 
$$
\pp{(\eta^{\alpha,\beta}_\ell(T_i)}_{i=1,\dots,m}
\to \pp{(\eta^{\alpha,\beta}(T_i)}_{i=1,\dots,m}
$$
a.s.~as $\ell\to\infty$, it only remains to check that for any
$i=1,\ldots, m$, 
\begin{multline} \label{e:last.step}
\lim_{\ell\to\infty} \limsup_{n\to\infty} \proba\left(
\frac{1}{b_n}\left|  \max_{k\in   nT_i}  \summ
  j1\infty\varepsilon_j G\bigl( \Gamma_j /2b_n^\alpha\bigr)
 {\bf 1}_{\big\{ T^k(U_j\topp n)_0=i_0\big\}} \right. \right.
\\
- \left.\left. \max_{k\in   nT_i}  \summ
  j1\ell\varepsilon_j  G\bigl( \Gamma_j /2b_n^\alpha\bigr)
\inddd{ T^k(U_j\topp n)_0=i_0}
\right|>\varepsilon\right)=0 
\end{multline}
for any $\varepsilon>0$. However, the above probability dos not exceed
\begin{align*}
&\proba\left(\frac{1}{b_n}
\left| \max_{k\in   nT_i}  \summ
  j{\ell+1}\infty\varepsilon_j G\bigl( \Gamma_j /2b_n^\alpha\bigr)
\inddd{ T^k(U_j\topp n)_0=i_0}
\right|>\varepsilon\right) \\
\leq &\proba\pp{\frac{1}{b_n}\max_{k=0,\dots,n}\abs{\sif
    j{\ell+1}\varepsilon_j
G\bigl( \Gamma_j /2b_n^\alpha\bigr)
\inddd{ T^k(U_j\topp n)_0=i_0}
}> \varepsilon} \,,
\end{align*}
and \eqref{e:last.step} follows from Proposition
\ref{prop:SM_remainder}. 
\end{proof}
 
  \subsection*{Acknowledgments} 
The authors are grateful to Shuyang Bai and Takashi Owada for pointing
out to us mistakes in an earlier version of the paper, and to 
Alexey Kuznetsov for helping us with the proof of Proposition
\ref{prop:first_intersection} by using special functions, in
particular for the reference \citet{prudnikov90integrals}. 
 
 \appendix
\section{Elements of renewal theory}\label{sec:renewal}

Consider an $\N$-valued renewal process $S_0 = 0$, $S_n:=
Y_1+\cdots+Y_n$, $n=1,2,\ldots$ 
whose inter-renewal times $\indn Y$ have a tail distribution $\wb F(n)
:= \proba(Y_1>n)$.  The  renewal function is defined by 
$u(n) := \sif k0\proba(S_k = n)$, and we denote $U(n):= \summ j0n u(j)$.
The following two assumptions are equivalent for all $\beta\in(0,1)$:
as $n\to\infty$, 
\begin{align}\label{eq:1}
\wb F(n) & \sim n^{-\beta}L(n), \\
U(n)& \sim \frac{n^\beta}{\Gamma(1+\beta)\Gamma(1-\beta)L(n)}\label{eq:2}.
\end{align}
See for example \citet[Theorem 8.7.3]{bingham87regular}. By Karamata's
theorem,  
\equh
u(n)  \sim \frac{n^{\beta-1}}{\Gamma(\beta)\Gamma(1-\beta)L(n)}
\label{eq:3}
\eque
implies \eqref{eq:2}. Furthermore, for $p_n:= \proba(Y_1 = n)$, under the  assumption
\equh\label{eq:Doney0}
\sup_{n\in\N}\frac{np_n}{\wb F(n)}<\infty,
\eque
it is known that  \eqref{eq:2} and
\eqref{eq:3} are equivalent; see \citet{doney97onesided}. 

Let now $\{Y_n\topp j\}_{n\in\N}, \, j=1,\dots,m$ be  independent
$\N$-valued renewal processes satisfying \eqref{eq:1} with 
with parameters $\beta_1,\dots,\beta_m\in(0,1)$ respectively, such that
\equh\label{eq:beta'}
\beta^*:= \summ q1m\beta_q - m+1\in(0,1),
\eque
and define $S_0\topp j := 0$, $S\topp j_n:= Y_1\topp j+\cdots+Y_n\topp
j$,  $n\geq 1$. Set 
\[
Y_1^*:= \min\ccbb{\ell\in\N: \ell = S_{n_k}\topp k \mbox{ for some }
  n_k\in\N, \forall k=1,\dots,m},
\]
and iteratively
\[
Y_{m+1}^*
:= \min\ccbb{\ell\in\N: Y_1^*+\cdots+ Y_m^*+\ell = S_{n_k}\topp k
  \mbox{ for some } n_k\in\N, \forall  k=1,\dots,m}.
\]
That is, $\indn{Y^*}$ are the simultaneous renewal times of
$\indn{Y\topp j}, j=1,\dots,m$, and they form another renewal process,
which we refer to as the {\em intersection renewal process} of $m$
independent renewal processes.  We denote by $F^*, u^*$ and $U^*$  the
corresponding functions defined at beginning of this section. 
\begin{Lem}\label{lem:1}
Assume that for every $j=1,\dots,m$,
\[
\wb F\topp j(n) = n^{-\beta_j}L_j(n), \beta_j\in(0,1) 
\]
for some slowly varying at infinity  function $L_j(n)$, and that
\eqref{eq:Doney0} holds. If $\beta_1,\dots,\beta_m$ satisfy
\eqref{eq:beta'}, then the intersection renewal process satisfies, as
$n\to\infty$,  
\[
\wb{F^*}(n) \sim n^{-\beta^*}L^*(n) \qmwith L^*(n) = \frac{\prodd q1m[\Gamma(\beta_q)\Gamma(1-\beta_q)L_q(n)]}{\Gamma(\beta^*)\Gamma(1-\beta^*)}.
\]
\end{Lem}
\begin{proof}
Let $u\topp q$ denote the renewal  function of the renewal process
$\{Y\topp q_n\}_{n\in\N}$, $q=1,\dots,m$. Because of \eqref{eq:Doney0}
we know that \eqref{eq:3} holds for each $q$. By independence, 
\[
u^*(n) = \prodd q1mu\topp q(n) \sim \frac{n^{\beta^*-1}}{\prodd
  q1m[\Gamma(\beta_q)\Gamma(1-\beta_q)L_q(n)]}. 
\]
Since \eqref{eq:3} implies \eqref{eq:1}, the desired result follows.
\end{proof}

\section{First intersection time of  independent shifted stable regenerative sets}

This paper uses certain results on the  intersections of
shifted stable regenerative sets.  Some of these
results may be known, but we could not find an 
appropriate reference. So we present them in this section. 

Let $B_{a,\beta}$ denote the overshoot distribution of a
$\beta$-stable subordinator over $a>0$. Recall that the density
function of $B_{a,\beta}$ in \eqref{e:ov.density}, 
and that the closure of image of a $\beta$-stable subordinator is
known as a $\beta$-stable regenerative set. Consider two independent
stable regenerative sets $R\topp{\beta_1}$ and $R\topp{\beta_2}$, with
indices $\beta_1,\beta_2$, respectively. We know that if $\beta_{1,2} :=
\beta_1+\beta_2-1>0$, then $R\topp{\beta_1} \cap R\topp{\beta_2}$ is
again a stable regenerative set, with parameter $\beta_{1,2}$.  

We will derive  an explicit formula for the cumulative distribution
function of the first intersection time of the two 
independent stable regenerative sets, with the second one shifted by
$a>0$. This random time is defined as  
\equh\label{eq:D_def}
D_{a,\beta_1,\beta_2} := \min\ccbb{R\topp{\beta_1}\cap \pp{a+R\topp{\beta_2}}}- a = \min\ccbb{\pp{R\topp{\beta_1}-a}\cap R\topp{\beta_2}}.
\eque
 
\begin{Thm}\label{thm:intersection_time}
For all $\beta_1,\beta_2\in(0,1)$ such that $\beta_1+\beta_2-1>0$, 
\equh\label{eq:recursion1}
D_{1,\beta_1,\beta_2} \eqd B_{1,\beta_1}(1+D_{1,\beta_2,\beta_1})
\eque
and
\equh\label{eq:scaling}
\frac{D_{a,\beta_1,\beta_2}}a \eqd D_{1,\beta_1,\beta_2} \qmfa a>0.
\eque
Moreover,
\equh\label{eq:D_CDF}
\proba(D_{1,\beta_1,\beta_2}\le x) = P_D^{\beta_1,\beta_2}(x\mid 1), \
\ x>0,,
\eque
where for $a>0$,
\equh
P_D^{\beta_1,\beta_2}(x\mid a) 
 = \frac1{\Gamma(\beta_{1,2})\Gamma(1-\beta_{1,2})}\int_0^1\pp{\frac
   ax+y}^{\beta_1-1}y^{\beta_2-1}(1-y)^{-\beta_{1,2}}dy. \label{eq:PD}
\eque
\end{Thm}
The main ingredient of the proof is the following proposition.
\begin{Prop}\label{prop:first_intersection} For all
  $\beta_1,\beta_2\in(0,1)$ such that $\beta_1+\beta_2-1>0$,    
\equh\label{eq:recursion_CDF}
P^{\beta_1,\beta_2}_D(x\mid 1) = \int_0^xp_B\topp{\beta_1}(y\mid
1)P_D^{\beta_2,\beta_1}(x-y\mid y) dy, \ \ x>0.
\eque
\end{Prop}
We first show how to derive Theorem \ref{thm:intersection_time} from this proposition.
\begin{proof}[Proof of Theorem \ref{thm:intersection_time}]
It follows from Lemma \ref{lem:intersection} that 
\[
D_{1,\beta_1,\beta_2} \eqd B_{1,\beta_1,0} \sum_{n=0}^\infty \pp{\prodd q1{n} B_{1,\beta_1,q}B_{1,\beta_2,q}+ B_{1,\beta_2,1}\prodd q1n B_{1,\beta_1,q}B_{1,\beta_2,q+1}},
\]
with the convention $\prodd q10 = 1$, 
where on the right-hand side $\{B_{1,\beta_i,n}\}_{n\in\N_0}$ are
i.i.d.~copies of $B_{1,\beta_i}$, $i=1,2$,
$\{B_{1,\beta_1,n}\}_{n\in\N_0}$ and $\{B_{1,\beta_2,n}\}_{n\in\N}$
are independent, and the series converges almost surely. This implies
the recursive relation \eqref{eq:recursion1}. Furthermore,
\eqref{eq:scaling} follows from \eqref{eq:D_def} and the scaling
invariance of the regenerative sets. 

It remains to prove \eqref{eq:D_CDF}. It follows from \eqref{eq:recursion1}  that
\equh\label{eq:recursion2}
D_{1,\beta_1,\beta_2} \eqd
B_{1,\beta_1}\bb{1+B_{1,\beta_2}(1+D_{1,\beta_1,\beta_2})}. 
\eque
By the Letac principle \citep{letac86contraction} applied to the recursion
\[
D_n =B_{1,\beta_1,n}[1+B_{1,\beta_2,n}(1+D_{n-1})], n\in\N,
\]
the law of $D_{1,\beta_1,\beta_2}$ satisfying 
\eqref{eq:recursion2}   is uniquely determined. Therefore,  it
suffices to show that a random variable whose law is given by the
right hand side of \eqref{eq:D_CDF} satisfies
\eqref{eq:recursion1}, that is that 
\[
P_D^{\beta_1,\beta_2}(x\mid 1) = \int_0^xp_B\topp{\beta_1}(y\mid 1) P_D^{\beta_2,\beta_1}\pp{\frac xy-1\mmid 1}dy, \mfa x>0.
\]
By the scaling property \eqref{eq:scaling} this is exactly
\eqref{eq:recursion_CDF}. 
\end{proof} 
\begin{proof}[Proof of Proposition \ref{prop:first_intersection}] 
Recall that the hypergeometric function  ${}_2F_1$ is defined as
\[
{}_2F_1(a,b;c;z) = \sif n0\frac{(a)_n(b)_n}{(c)_n}{\frac{z^n}{n!}}
\]
(for $c$ that is not a nonpositive integer) 
with $(a)_0 = 0$ and $(a)_n = a(a+1)\cdots(a+n-1)$ for $n\in\N$. 
By the Euler integral representation 
\[
B(b,c-b){}_2F_1(a,b;c;z) = \int_0^1x^{b-1}(1-x)^{c-b-1}(1-zx)^{-a}dx,
\]
for ${\rm Re}(c) > {\rm Re}(b)>0$ and $z\in\C\setminus[1,\infty)$ ($B$
is the beta function), we have 
\[
P_D^{\beta_1,\beta_2}(x\mid a) =
\frac{\Gamma(\beta_2)}{\Gamma(\beta_{1,2})\Gamma(2-\beta_1)}\pp{\frac
  ax}^{\beta_1-1}{}_2F_1\pp{1-\beta_1,\beta_2;2-\beta_1;-\frac xa}. 
\]
Therefore, the right-hand side of   \eqref{eq:recursion_CDF} is 
\begin{multline*}
\frac1{\Gamma(\beta_{1,2})\Gamma(2-\beta_2)\Gamma(1-\beta_1)}\\
\times \int_0^x\frac1{1+y}\frac1{y^{\beta_1}}\pp{\frac y{x-y}}^{\beta_2-1}{}_2F_1\pp{1-\beta_2,\beta_1;2-\beta_2;1-\frac xy}dy.
\end{multline*}
Changing the variable $u = x/y-1$, the above becomes
\[
\frac{x^{1-\beta_1}}{\Gamma(\beta_{1,2})\Gamma(2-\beta_2)\Gamma(1-\beta_1)} \int_0^\infty \frac {u^{1-\beta_2}}{u+x+1}(1+u)^{\beta_1-1}{}_2F_1\pp{1-\beta_2,\beta_1;2-\beta_2;-u}du.
\]
By the Euler transformation
\[
(1+u)^{\beta_1-1}{}_2F_1\pp{1-\beta_2,\beta_1;2-\beta_2;-u} = F_1\pp{1,1-\beta_{1,2};2-\beta_2;-u},
\]
 this can be written as 
\equh\label{eq:hypergeometric1}
\frac{x^{1-\beta_1}}{\Gamma(\beta_{1,2})\Gamma(2-\beta_2)\Gamma(1-\beta_1)} \int_0^\infty \frac {u^{1-\beta_2}}{u+x+1}{}_2F_1\pp{1,1-\beta_{1,2};2-\beta_2;-u}du.
\eque
Using the table of integrals of
\citet[2.21.1.16]{prudnikov90integrals}, 
\begin{multline*}
\int_0^\infty\frac{x^{c-1}}{(x+z)^\rho}{}_2F_1(a,b;c;-wx)dx\\
= w^{a-c}\frac{\Gamma(c)\Gamma(a-c+\rho)\Gamma(b-c+\rho)}{\Gamma(a+b-c+\rho)}{}_2F_1(a-c+\rho,b-c+\rho;a+b-c+\rho;1-wz)
\end{multline*}
(provided ${\rm Re}(a+\rho), {\rm Re}(b+\rho)>{\rm Re}(c)>0, |\arg
w|,|\arg z|<\pi)$), \eqref{eq:hypergeometric1}  becomes 
\[
\frac{\Gamma(\beta_2)}{\Gamma(\beta_{1,2})\Gamma(2-\beta_1)}x^{1-\beta_1}{}_2F_1(\beta_2,1-\beta_1;2-\beta_1;-x) = P_D^{\beta_1,\beta_2}(x\mid 1). 
\]
This completes the proof.
\end{proof}
\begin{Coro}\label{coro:1}
Let $\ell\in\N, \ell\ge2$ and $\beta_1,\dots,\beta_\ell\in(0,1)$ be such that
\[
\beta^*_\ell:=\summ j1\ell\beta_j-\ell+1 > 0.
\]
For each $j=1,\dots,\ell$, let $R\topp{\beta_j}_j$ be a $\beta_j$-stable
regenerative set and $V\topp{\beta_j}_j$   a random variable with
$\proba(V\topp{\beta_j}_j\le x) = x^{1-\beta_j}, \, x\in(0,1)$. Assume that all $R\topp{\beta_j}_j, V\topp{\beta_j}_j, j=1,\dots,\ell$ are independent. 
 Then,
 \[
 \bigcap_{j=1}^\ell\pp{V\topp{\beta_j}_j+R\topp{\beta_j}_j} \eqd \wt V\topp\ell+R\topp{\beta_\ell^*},
 \]
where $R\topp{\beta^*_\ell}$ is a $\beta^*_\ell$-stable regenerative
set, independent of a  nonnegative random variable $\wt V\topp\ell$,
whose law satisfies 
 \equh\label{eq:Vt}
 \proba\pp{\wt V\topp\ell\le x} =
 \frac{x^{1-\beta^*_\ell}}{\Gamma(\beta^*_\ell)\Gamma(2-\beta^*_
   \ell)} \prodd j1\ell \pp{\Gamma(\beta_j)\Gamma(2-\beta_j)} \
 \text{for} \  x\in[0,1]. 
 \eque 
\end{Coro}
\begin{proof}
The proof is by induction in $\ell$. For $\ell=1$ the claim is
trivial. 
We proceed with the case $\ell = 2$. We already know that the intersection
 interest has the law of a $\beta^*_{2}$-stable regenerative
set shifted by an independent random variable, so it 
suffices to check that the law of the shift satisfies \eqref{eq:Vt}. 
We will use \eqref{eq:PD} in the form (obtained by a change of
variables) 
\[
P_D^{\beta_1,\beta_2}(x\mid a) =\frac1{\Gamma(\beta_{2}^*)\Gamma(1-\beta_{2}^*)} \int_0^x(a+y)^{\beta_1-1}y^{\beta_2-1}(x-y)^{-\beta_{1,2}}dy.
\]
Fix $x\in (0,1)$. Then 
\begin{multline}
\label{eq:V*}
\proba\pp{\wt V\topp 2\le x} = \int_0^x
(1-\beta_1)y_1^{-\beta_1}\int_0^{y_1}(1-\beta_2)y_2^{-\beta_2}P_D^{\beta_2,\beta_1}(x-y_1\mid
y_1-y_2)dy_2dy_1 \\
+ \int_0^x
(1-\beta_1)y_1^{-\beta_1}\int_{y_1}^x(1-\beta_2)y_2^{-\beta_2}P_D^{\beta_1,\beta_2}(x-y_2\mid
y_2-y_1)dy_2dy_1, 
\end{multline}
Denote 
\[
c_{\beta_1,\beta_2} := \frac{(1-\beta_1)(1-\beta_2)}{\Gamma(\beta^*_2)\Gamma(1-\beta^*_2)}.
\]
Then the first double integral in \eqref{eq:V*} can be reduced, after
a change of variable, to 
\[
 c_{\beta_1,\beta_2} \int_0^x\int_0^{y_1}\int_{y_1}^x
 y_1^{-\beta_1}y_2^{-\beta_2}(z-y_2)^{\beta_2-1}(z-y_1)^{\beta_1-1}(x-z)^{-\beta^*_2}dzdy_2dy_1,
\]
while the second double integral in \eqref{eq:V*} reduces, similarly,
to 
\[
c_{\beta_1,\beta_2} \int_0^x\int_{y_1}^x\int_{y_2}^x y_1^{-\beta_1}y_2^{-\beta_2}(z-y_1)^{\beta_1-1}(z-y_2)^{\beta_2-1}(x-z)^{-\beta^*_2}dzdy_2dy_1.
\]
By Fubini theorem, for any nonnegative function $f$, 
\[
\int_0^x\int_0^{y_1}\int_{y_1}^x f
dzdy_2dy_1+\int_0^x\int_{y_1}^x\int_{y_2}^xfdzdy_2dy_1
=\int_0^x\int_0^z\int_0^z fdy_1dy_2dz. 
\]
Therefore, 
\begin{align*}
&\proba\pp{\wt V\topp 2\le x} \nonumber
\\&=  c_{\beta_1,\beta_2}\int_0^x\int_0^z\int_0^zy_1^{-\beta_1}y_2^{-\beta_2}(z-y_1)^{\beta_1-1}(z-y_2)^{\beta_2-1}(x-z)^{-\beta^*_2}dy_1dy_2dz\nonumber\\
& = c_{\beta_1,\beta_2}B(1-\beta_1,\beta_1)B(1-\beta_2,\beta_2)\int_0^x(x-z)^{-\beta^*_2}dz,
\end{align*}
which is the same as \eqref{eq:Vt} with $\ell = 2$.

Next, suppose that the claim holds for some $\ell\geq 2$, and consider
the intersection of $\ell+1$ independent shifted stable regenerative
sets. By the assumption of the induction, 
\[
\bigcap_{j=1}^{\ell+1}\pp{V_j\topp{\beta_j}+R_j\topp{\beta_j}}  \eqd \pp{\wt V\topp\ell+R\topp{\beta^*_\ell}}\cap \pp{V_{\ell+1}\topp{\beta_{\ell+1}}+R_{\ell+1}\topp{\beta_{\ell+1}}},
\]
where the four random elements on the right-hand side above are
assumed to be independent. Again the intersection on the right-hand
side above, by strong Markov property, is a shifted stable
regenerative set with index $\beta^*_\ell+\beta_{\ell+1}-1 = \beta^*_{\ell+1}$, and it
suffices to identify the law of the shift $\wt V\topp {\ell+1}$.  

Let $x\in (0,1)$. We have
\begin{align*}
\proba\pp{\wt V\topp {\ell+1}\le x} & = \proba\pp{\wt V\topp {\ell+1}\le
  x, \, \wt V\topp {\ell}\le
  1} \\
& =  \proba\pp{\wt V\topp {\ell+1}\le
  x \Big|\wt V\topp {\ell}\le   1} \proba\pp{\wt V\topp {\ell}\le 1}\,.
\end{align*}
By the assumption of the induction, the conditional cumulative distribution function of
$\wt V\topp {\ell}$ given $\wt V\topp {\ell}\le   1$ is
$x^{1-\beta^*_\ell}, \, 0\leq x\leq 1$. Therefore, we are in the
situation of the intersection of $2$ independent shifted stable regenerative
sets, which has already been considered. Using once again the
assumption of the induction, we obtain
\begin{align*}
\proba\pp{\wt V\topp {\ell+1}\le x} 
 = &
  \frac{x^{1-\beta^*_{\ell+1}}}{\Gamma(\beta^*_{\ell+1})\Gamma(2-\beta^*_{\ell+1})} 
\Bigl( \Gamma(\beta^*_{\ell})\Gamma(2-\beta^*_{\ell})
                                      \Gamma(\beta_{\ell+1})\Gamma(2-\beta_{\ell+1})\Bigr)
  \\ 
& \times \frac{1}{\Gamma(\beta^*_\ell)\Gamma(2-\beta^*_
   \ell)} \prodd j1\ell \pp{\Gamma(\beta_j)\Gamma(2-\beta_j)}
  \\
= &
    \frac{x^{1-\beta^*_{\ell+1}}}{\Gamma(\beta^*_{\ell+1})\Gamma(2-\beta^*_{\ell+1})}  
\prodd j1{\ell+1} \pp{\Gamma(\beta_j)\Gamma(2-\beta_j)}\,,
\end{align*}
as required. 
\end{proof}

\bibliographystyle{apalike}
\bibliography{bibfile,references}
\end{document}